\documentclass[reqno,twoside,11pt]{article}
\usepackage{amssymb,amsfonts,amsthm,amsmath}
\usepackage[pagebackref=false]{hyperref}

\usepackage[hmargin=1in,vmargin=1in]{geometry}

\usepackage{cite}

\usepackage{color}

\def\eqdef{\stackrel{\rm def}{=}}

\def\beq{\begin{equation}}
\def\eeq{\end{equation}}
\def\beqs{\begin{equation*}}
\def\eeqs{\end{equation*}}

\newtheorem{theorem}{Theorem}[section]
\newtheorem{lemma}[theorem]{Lemma}
\newtheorem{proposition}[theorem]{Proposition}
\newtheorem{corollary}[theorem]{Corollary}

\theoremstyle{definition}

\newtheorem{example}[theorem]{Example}

\def\myclearpage{}

\definecolor{darkred}{rgb}{.70,.12,.20}

\definecolor{darkgreen}{rgb}{.20,.52,.14}

\newcommand{\esssup}{\mathop{\mathrm{ess\,sup}}}

\newcommand{\supp}{\operatorname{supp}}

\numberwithin{equation}{section}

\title{Maximum estimates for generalized Forchheimer flows in heterogeneous porous media}
\author{Emine Celik and Luan Hoang}
\date{\today}

\begin{document}

\maketitle

\begin{center}
Department of Mathematics and Statistics, Texas Tech University, Box 41042, Lubbock, TX 79409--1042, U.S.A.\\
Email addresses:  \texttt{emine.celik@ttu.edu}, \texttt{luan.hoang@ttu.edu}\\
\end{center}

\begin{abstract}
This article continues the study in  \cite{CH1} of generalized Forchheimer flows in heterogeneous porous media. 
Such flows are used to account for deviations from Darcy's law.
In heterogeneous media, the derived nonlinear partial differential equation for the pressure can be singular and degenerate in the spatial variables, in addition to being degenerate for large pressure gradient. Here we obtain the estimates for the $L^\infty$-norms of the pressure and its time derivative in terms of the initial and  the time-dependent boundary data. 
They are established by implementing De Giorgi's iteration in the context of weighted norms with the weights specifically defined by the Forchheimer equation's coefficient functions. 
With these weights, we prove suitable weighted parabolic Poincar\'e-Sobolev inequalities and use them to facilitate the iteration.
Moreover, local in time $L^\infty$-bounds are combined with uniform Gronwall-type energy inequalities to obtain long-time $L^\infty$-estimates.
\end{abstract}


\pagestyle{myheadings}\markboth{E. Celik and L. Hoang}
{Maximum Estimates for Generalized Forchheimer Flows}

\myclearpage
\section{Introduction}\label{intro}

Studies of fluid flows in porous media usually use the  Darcy equation as a law. 
However, when the Reynolds number is large, this linear equation is not accurate anymore in describing the fluid dynamics.
In that case, Forchheimer equations \cite{Forchh1901,ForchheimerBook} are commonly used instead. 
Unlike Darcy's equation, these are nonlinear relations between the velocity and pressure gradient. 
They are also proposed as models for turbulence in porous media, see e.g. \cite{Ward64}.
The reader is referred to \cite{ABHI1,HI2} and \cite{MuskatBook,BearBook,StraughanBook,NieldBook}
for more information about the Forchheimer flows and their generalizations.

Compared to the Darcy flows, mathematical analysis of the Forchheimer models is scarce.
Moreover, previous mathematical works on Forchheimer flows only consider the homogeneous porous media, see e.g. \cite{Payne1999a,Straughan2010} for incompressible fluids,   \cite{ABHI1,HI1,HI2,HIKS1,HKP1} for slightly compressible fluids, and \cite{CHK1} for  isentropic gases. 
The  problem of Forchheimer flows in heterogeneous media, which is encountered frequently in real life applications,  was started in \cite{CH1}.
The current article is a continuation of \cite{CH1} and is focused on the $L^\infty$-estimates rather than $L^2$.
Below, we follow \cite{CH1} in presenting the model and deriving the key partial differential equation (PDE).

Let a porous medium be modeled as a  bounded domain $U$ in space $\mathbb{R}^n$ with $C^1$-boundary $\Gamma=\partial U$. Throughout this paper, $n\ge 2$ even though for physics problems $n=2$ or $3$.
Let $x\in \mathbb{R}^n$ and $t\in \mathbb{R}$ be the spatial and time variables.
The porosity of this heterogeneous media is denoted by $\phi=\phi(x)$ which  depends on the location $x$.

For a fluid flow in the media, we denote the velocity by $v(x,t)\in \mathbb{R}^n,$ pressure by $p(x,t)\in \mathbb{R}$ and density by $\rho(x,t)\in \mathbb{R}^+=[0,\infty)$.

A generalized Forchheimer equation is
 \beq\label{eq1}
g(x,|v|)v=-\nabla p,
\eeq
where $g(x,s)\ge 0$ is a function defined on $\bar U\times  \mathbb{R}^+$.
Here, we focus on the case when the function $g$ in \eqref{eq1} is  of the form
\beq\label{eq2}
g(x,s)=a_0(x)s^{\alpha_0} + a_1(x)s^{\alpha_1}+\cdots +a_N(x)s^{\alpha_N}\quad\text{for } s\ge 0,
\eeq
where $N\ge 1,\alpha_0=0<\alpha_1<\cdots<\alpha_N$ are fixed real numbers, the coefficient functions  $a_1(x)$, $a_2(x)$, \ldots, $a_{N-1}(x)$ are non-negative,  and $a_0(x),a_N(x)$ are positive. 
The number $\alpha_N$ is the degree of $g$ and is denoted by $\deg(g)$. 

Equation \eqref{eq1} with $g$ defined by \eqref{eq2}  is a generalization of Darcy and Forchheimer equations \cite{ABHI1,HI1,HI2}.
For instance, when
\beq\label{simg} 
g(x,s)=\alpha,\
 \alpha+\beta s,\
  \alpha+\beta s+\gamma s^2,\
   \alpha +\gamma_m s^{m-1}, 
\eeq 
where $\alpha$, $\beta$, $\gamma$, $m\in(1,2]$, $\gamma_m$ are empirical constants, we have Darcy's law, Forchheimer's two term, three term and power laws, respectively, for homogeneous media, see e.g.~\cite{MuskatBook,BearBook}. 
The dependence of $a_i$'s on $x$ indicates the media being heterogeneous.
The case when $a_i(x)$'s  are independent of $x$ was studied in depth in \cite{HI1,HI2,HIKS1,HKP1,HK2}.

From \eqref{eq1} one can solve for $v$ in terms of $\nabla p$ and obtain the equation
\beq\label{eq3}
v=-K(x,|\nabla p|)\nabla p,
\eeq
where the function $K:\bar U\times \mathbb{R}^+\to\mathbb{R}^+$ is defined by
\beq\label{eq4}
K(x,\xi)=\frac{1}{g(x,s(x,\xi))} \quad\text{for } x\in \bar U,\ \xi\ge 0,
\eeq
with $s=s(x,\xi)$ being the unique non-negative solution of  $sg(x,s)=\xi$.

We combine  \eqref{eq3} with the equation of continuity
\beqs
\phi\frac{\partial \rho}{\partial t} +\nabla\cdot(\rho v)=0,
\eeqs
   and the equation of state which, for (isothermal)  slightly compressible fluids, is
   \beqs
\frac{1}{\rho}   \frac{d\rho}{dp}=\varpi, 
\quad \text{where the constant compressibility } \varpi>0 \text { is small}.
   \eeqs

With small $\varpi$, by a slight simplification and time scaling, we derive the following initial boundary value problem (IBVP) for the pressure $p(x,t)$:
\begin{equation}\label{ppb}
\begin{cases}
\phi(x) \begin{displaystyle}\frac{\partial p}{\partial t}\end{displaystyle}
= \nabla \cdot (K(x,|\nabla p|)\nabla p) \quad \text{on } U\times(0, \infty),\\
p=\psi \quad \text{on} \quad \Gamma \times (0, \infty),\\
p(x,0)=p_0(x) \quad \text{on}\quad   U,
\end{cases}
\end{equation} 
where  $p_0(x)$ are $\psi(x,t)$ are given initial and boundary data. (See \cite{CH1} for more details.)

Here afterward, the function $g(x,s)$ in \eqref{eq2} is fixed, hence so is $K(x,\xi)$.

Although $\phi(x)$ belongs to $(0,1)$ in applications, we only assume $\phi(x)>0$ in this paper.

As noted in \cite{CH1}, the PDE in \eqref{ppb} is degenerate in $\nabla p$ as $|\nabla p|\to\infty$, and can be both singular and degenerate in $x$. For such a nonlinear PDE, finer analysis is needed to deal with different types of degeneracy and singularity. To obtain maximum estimates for the solutions, De Giorgi's iteration is used with suitable weighted norms. Thanks to the structure of our equation, these weights are properly defined based on the functions $\phi(x)$ and $a_i(x)$'s. For such weights, the corresponding weighted energy and gradient estimates were already established in \cite{CH1}. It turns out that we can obtain the maximum estimates for both $p$ and its time derivative under a slightly more stringent condition than the one imposed in \cite{CH1}, see \eqref{P2} compared to \eqref{P1} below. Then the $L^\infty$-estimates for large time are derived with the use of the uniform Gronwall-type inequalities from \cite{CH1}. 

The paper is organized as follows.
In section \ref{auxsec}, we establish suitable weighted parabolic Poincar\'e-Sobolev inequalities.
In section \ref{review}, we review essential results from \cite{CH1} that will be needed for the current work.
Sections \ref{psec} and \ref{ptsec} contain estimates of the  $L^\infty$-norm for $p$ and $\partial p/\partial t$.
Local in time estimates are established in Propositions \ref{Linftee} and \ref{theo51} by De Giorgi's iteration using appropriate weighted norms and the corresponding Poincar\'e-Sobolev inequalities in section \ref{auxsec}.
The main results  in terms of initial and boundary data are obtained  in Theorems \ref{mainp} and \ref{theo52}.
Particularly, the asymptotic estimates as time goes to infinity are improved to depend only on the asymptotic behavior of the boundary data. This is done by  combining the local in time estimates with uniform Gronwall-type inequalities. 
 It is worth mentioning that our results are applicable to all commonly used Forchheimer's laws.
Finally, we remark that in case of homogeneous porous media, estimates for $p$ and its time derivative pave the way for obtaining $L^\infty$-estimates for the gradient, as well as strong continuous dependence  and structural stability, see  \cite{HKP1,HK1,HK2}. However, it is not known whether such results still hold true for heterogeneous media in the current study.


\myclearpage
\section{Auxiliaries}\label{auxsec}
First, we recall some elementary inequalities that will be needed. 
Let $x,y\ge 0$, then
\beq\label{ee2}
(x+y)^p\le x^p+y^p\quad  \text{for all } 1\ge p>0,
\eeq
\beq\label{ee3}
(x+y)^p\le 2^{p-1}(x^p+y^p)\quad  \text{for all }  p\ge 1,
\eeq
\beq\label{ee4}
x^\beta \le x^\alpha+x^\gamma\quad \text{for all } \gamma\ge \beta\ge\alpha\ge 0,
\eeq
\beq\label{ee5}
x^\beta \le 1+x^\gamma \quad \text{for all } \gamma\ge \beta\ge 0.
\eeq
Also,
\beq\label{ee7}
|x-y|^p\ge 2^{-p+1}|x|^p-|y|^p\quad  \text{for all }  x,y\in\mathbb R^n\text{ and } p\ge 1.
\eeq

We establish below some weighted parabolic Poincar\'e-Sobolev inequalities which are suitable to the PDE in \eqref{ppb} and are essential to our $L^\infty$-estimates.

We recall the standard Sobolev-Poincar\'e's inequality.
Let $\mathring W^{1,q}(U)$ be the space of functions in $W^{1,q}(U)$ with vanishing traces on the boundary.
If $1\le q<n$ then 
\beq\label{PS0} \| f\|_{L^{q^*}(U)} \le c \| \nabla f\|_{L^q(U)} \text{ for all }f\in\mathring{W}^{1,q}(U),
\eeq
where $q^*=n q/(n-q)$, the positive constant $c$ depends on $q$, $n$ and the domain $U$. For our problem, we need some weighted versions of this.

If $f(x)\ge 0$ is a $L^1$-function on $U$, then define a measure $\mu_f$ on $U$  by
\beqs
d\mu_f=f(x) dx.
\eeqs
For any $p\in[1,\infty]$ and a measurable set $E\subset U$, we denote by $L^p_f(E)$ and $\|\cdot\|_{L^p_f(E)}$ the  $L^p$ space and, respectively, the $L^p$ norm on $E$ corresponding to the measure $\mu_f$.

Similarly, if $f(x,t)\ge 0$ on $U\times\mathbb R$ satisfies  $f\in L^1(U\times(t_1,t_2))$ for any real numbers $t_1<t_2$, then define a measure  $\bar \mu_f$ on $U\times\mathbb R$ by
\beq\label{meas2}
d\bar\mu_f=f(x,t)dxdt.
\eeq
For any $p\in[1,\infty]$ and a bounded, measurable set $E\subset U\times \mathbb R$, we denote by $L^p_f(E)$ and $\|\cdot\|_{L^p_f(E)}$ the  $L^p$ space and, respectively, the $L^p$ norm on $E$ corresponding to the measure  $\bar\mu_f$.

Let $\gamma_1(x),\gamma_2(x)> 0$ be two functions on $U$. Here is the two-weight Poincar\'e-Sobolev inequality that we need: There is a positive constant $c_0$ such that
\beq\label{Sphigam}
\|u\|_{L^r_{\gamma_1}(U)}\le c_0 \|\nabla u\|_{L^q_{\gamma_2}(U)}
\eeq
for all $u$ belonging to a certain class $\mathring X^{r,q}_{\gamma_1,\gamma_2}(U)$ containing functions which vanish on the boundary $\Gamma$.

For some classes of $\gamma_1$, $\gamma_2$, and $\mathring X^{r,q}_{\gamma_1,\gamma_2}(U)$, see e.g. \cite{DTP,SW1992,Cianchi1996}.
For instance, \cite{Cianchi1996} characterizes $\gamma_1$ and  $\gamma_2$ so that \eqref{Sphigam} holds for all $u$ such that its extension to zero outside $U$ belongs to $W^{1,1}(\mathbb R^n)$. Of course, there are more than one  characterization and one class $\mathring X^{r,q}_{\gamma_1,\gamma_2}(U)$.
To avoid considerations of complicated weighted spaces, we will take \eqref{Sphigam} as our starting point.
In Example \ref{SDCrmk} below, we give simple examples for a few classes which are applicable to our particular problem.

Assume \eqref{Sphigam} holds for $\gamma_1(x)$, $\gamma_2(x)$ and a space $\mathring X^{r,q}_{\gamma_1,\gamma_2}(U)$.

For $T>0$, denote $Q_T=U\times (0,T)$ and 
\beq 
\mathring X^{r,q}_{\gamma_1,\gamma_2}(Q_T)\eqdef\Big \{ u(x,t): u(\cdot,t)\in \mathring X^{r,q}_{\gamma_1,\gamma_2}(U) \text{ for almost all } t\in(0,T)\Big\}.
\eeq 

Let $c_0$ be the positive constant in \eqref{Sphigam}.

Throughout, for convenience,  we denote $f(t)=f(\cdot,t)$ for any function $f(x,t)$.

\begin{lemma}\label{PSob1}
Let $r,q$ be two numbers satisfying
\beq\label{qr}
r>2,\quad r>q\ge 1.\eeq
Set 
\beq \label{pcond}
p=2+q(1-2/ r)=q+2(1-q/r).
\eeq

If $T>0$ and $u(x,t)\in \mathring X^{r,q}_{\gamma_1,\gamma_2}(Q_T)$, then 
\beq\label{PSi1}
\| u \|_{L_{\gamma_1}^{p}(Q_T)} \le c_0^\frac{q}{p}(\esssup_{0<t<T} \| u(t)\|_{L^2_{\gamma_1}(U)})^{1-\frac{q}{p}}\|\nabla u\|_{L^q_{\gamma_2}(Q_T)}^\frac{q}{p}.
\eeq   
Consequently,
\beq\label{PSi2}
\| u \|_{L_{\gamma_1}^{p}(Q_T)} \le c_0^\frac{q}{p} \Big(\esssup_{0<t<T} \|u(t)\|_{L_{\gamma_1}^2(U)}+\|\nabla u\|_{L^q_{\gamma_2}(Q_T)}\Big).
\eeq
\end{lemma}
\begin{proof}
Condition \eqref{qr} and definition \eqref{pcond} imply that $q<p$ and $2<p< r$. 
Let $\alpha=1-\frac{q}{p}$ and $\beta=\frac{q}{p}$.
Then $\alpha,\beta\in (0,1)$,   
\beq\label{betacond}
\alpha+\beta=1\quad \text{and}\quad \frac 1 {p}=\frac {\alpha}2+\frac {\beta}{ r}.
\eeq
Then by interpolation inequality and \eqref{Sphigam}, we have for almost all $t\in(0,T)$ that
\begin{align*}
\left( \int_U |u(t)|^{p} \gamma_1 dx \right)^{\frac 1{p}} 
&\le \left( \int_U |u(t)|^{2} \gamma_1 dx \right)^{\frac {\alpha} {2}}\left( \int_U | u(t)|^{r} \gamma_1 dx \right)^{\frac {\beta}{ r}}\\
& \le c_0^\beta\left( \int_U  |u(t)|^{2} \gamma_1 dx \right)^{\frac {\alpha} {2}}\left(\int_U |\nabla u(t)|^{q} \gamma_2 dx \right)^{\frac \beta{q}}.
\end{align*}
Taking the power $p$ of both side of the previous inequality and integrating it in $t$ from $0$ to $T$, we have 
\begin{align}
\int_0^T \int_U   |u|^{p} \gamma_1 dx dt 
&\le  c_0^{\beta p}\int_0^T\left( \int_U  |u(t)|^2 \gamma_1 dx \right)^{\frac {\alpha p}2}\left(\int_U |\nabla u(t)|^{q} \gamma_2  dx \right)^\frac {\beta p}{q}dt \notag\\
&\le  c_0^{\beta p}\left( \esssup_{0<t<T}\int_U  |u(t)|^2 \gamma_1 dx \right)^{\frac {\alpha p}2} \int_0^T\left(\int_U |\nabla u(t)|^{q} \gamma_2  dx \right)^\frac {\beta p}{q}dt.\label{estwphi}
\end{align}
Since $\beta p/q=1$, we obtain
\beq\label{pp}
\| u\|^{p}_{L_{\gamma_1}^{p}(Q_T)}\le c_0^{\beta p} \esssup_{0<t<T} \| u\|_{L^2_{\gamma_1}(U)}^{\alpha p}\left(\int_0^T\int_U |\nabla u|^{q} \gamma_2  dx dt\right)^\frac {\beta p}{q}.
\eeq
Taking power $1/p$ of \eqref{pp} yields \eqref{PSi1}.

In  \eqref{PSi1}, we bound 
$$\esssup_{0<t<T}\|u(t)\|_{L_{\gamma_1}^2(U)}\quad\text{and}\quad \|\nabla u\|_{L_{\gamma_2}^q(Q_T)}$$ by their sum,  then \eqref{PSi2} follows.
\end{proof}

We will refer to the following inequality as Strict Degree Condition (SDC)
\beq\label{SDC}
 \deg(g) < \frac{4}{n-2} .
\eeq

Note that in the three dimensional cases (n=3), \eqref{SDC} reads deg$(g)<4$, hence it holds for the commonly used Forchheimer models in \eqref{simg}.

\begin{example}\label{SDCrmk}
We give examples for the weighted elliptic Poincar\'e-Sobolev inequality \eqref{Sphigam}. The parabolic inequalities in Lemma \ref{PSob1}, hence, follow correspondingly.

(a) Suppose $q\in [1,n)$ and $r$ is a number in the interval $[1,q^*)$. 

Let $q_0\in [1,q)$ such that $ r<q_0^*<q^*$. 
Assume
\beq\label{gamscond}
 \int_U \gamma_1 (x)^\frac{q_0^*}{q_0^*- r}dx, \int_U \gamma_2(x)^{-\frac{q_0}{q-q_0}}dx<\infty.
\eeq

Let $u\in \mathring W^{1,q_0}(U)$. By H\"{o}lder's inequality with powers $\frac{q_0^*}{r}$ and $\frac{q_0^*}{q_0^*- r}$, we have 
\beq
\Big(\int_U|u|^{r} \gamma_1dx\Big)^{\frac 1{r}} \le \Big(\int_U|u|^{q_0^*}dx\Big)^{\frac 1{q_0^*}}\Big(\int_U \gamma_1^\frac{q_0^*}{q_0^*- r}dx\Big)^\frac {q_0^*- r}{q_0^* r}.
\eeq
Applying  \eqref{PS0} to the first Lebesgue norm on the right-hand side gives 
\beq
\Big(\int_U|u|^{ r} \gamma_1 dx\Big)^{\frac 1{r}} \le c\Big(\int_U|\nabla u|^{q_0}dx\Big)^\frac{1}{q_0}\Big(\int_U \gamma_1^\frac{q_0^*}{q_0^*- r}dx\Big)^{\frac{q_0^*-r}{q_0^* r}}.
\eeq
where $c$ is the constant in \eqref{PS0} with $q=q_0$. 
Since $q>q_0$, we bound the first integral on the right-hand side by applying H\"{o}lder's inequality to functions $|\nabla u|^{q_0}\gamma_2^\frac{q_0}{q}$ and $\gamma_2^{-\frac{q_0}{q}}$ with powers $q/q_0$ and $q/(q-q_0)$. We obtain 
\beqs
\Big(\int_U|u|^{ r} \gamma_1 dx\Big)^{\frac 1{r}} \le c_0\Big(\int_U|\nabla u|^{q}\gamma_2dx\Big)^\frac{1}{q}
\eeqs
with
\beq\label{Meg}
c_0=c\Big( \int_U \gamma_2(x)^{-\frac{q_0}{q-q_0}}dx \Big)^\frac{q-q_0}{q q_0}\Big(\int_U \gamma_1 (x)^\frac{q_0^*}{q_0^*- r}dx\Big)^\frac {q_0^*- r}{q_0^* r}<\infty.
\eeq

Therefore, \eqref{Sphigam} holds with   $c_0$ given by \eqref{Meg} and 
$$\mathring X^{r,q}_{\gamma_1,\gamma_2}(U)=\mathring W^{1,q_0}(U)\cap \{u:\nabla u\in L_{\gamma_2}^q(U)\}.$$ 

(b) In \cite{CH1}, we used the case $r=2$ and $q=2-a$. Condition \eqref{gamscond} becomes
\beqs
 \int_U \gamma_1 (x)^\frac{q_0^*}{q_0^*- 2}dx, \int_U \gamma_2(x)^{-\frac{q_0}{2-a-q_0}}dx<\infty,
\eeqs
where $1\le q_0<2-a$ such that $q_0^*>2$.

(c) We consider the case when $r>2$ and $q=2-a$. 
Assume (SDC). One can easily verify that  $2<  (2-a)^*$.
Suppose $ r$ is a number in the interval $(2,(2-a)^*)$. 
Let $q_0\in(1,2-a)$ such that $ r<q_0^*<(2-a)^*$. 
Assume
\beq\label{gams2}
 \int_U \gamma_1 (x)^\frac{q_0^*}{q_0^*- r}dx, \int_U \gamma_2(x)^{-\frac{q_0}{2-a-q_0}}dx<\infty.
\eeq
Then we obtain
\beqs
\|u\|_{L^r_{\gamma_1}(U)}\le c_0 \|\nabla u\|_{L^{2-a}_{\gamma_2}(U)}
\eeqs
with
\beqs
c_0=c\Big( \int_U \gamma_2(x)^{-\frac{q_0}{2-a-q_0}}dx \Big)^\frac{2-a-q_0}{(2-a) q_0}\Big(\int_U \gamma_1 (x)^\frac{q_0^*}{q_0^*- r}dx\Big)^\frac {q_0^*- r}{q_0^* r}<\infty.
\eeqs
\end{example}

\begin{lemma}\label{PSob2}
Let $r$, $q$, $ \gamma_1 (x)$, $\gamma_2(x)$, $c_0$ be the same as in Lemma \ref{PSob1}.
Let $m$ be a number in $(q,r)$, and define 
\beq \label{pdef}
p=2+m\Big( 1-\frac{2}{ r} \Big).
\eeq
Then for any $T>0$, a function  $F(x,t)>0$ on $Q_T$, and $u(x,t)\in \mathring X^{r,q}_{\gamma_1,\gamma_2}(Q_T)$, one has
\beq\label{gammaandK}
\begin{aligned}
\| u\|_{L_{\gamma_1}^{p}(Q_T)} \le  c_0^{\frac{m}{p}}\mathcal{Z}^\frac{m-q}{pq} \Big(\esssup_{0<t<T} \|u(t)\|_{L_{\gamma_1}^2(U)}+\| \nabla u\|_{L^m_F(Q_T)}\Big),
\end{aligned}
\eeq
where  
\beqs
\mathcal{Z}=  \esssup_{0<t<T}\left(\int_U \gamma_2(x)^{\frac{m}{m-q}}F(x,t)^{-\frac{q}{m-q}}\chi_{{\rm supp}u}(x,t)dx\right).
\eeqs
\end{lemma}
\begin{proof}
Denote
\beqs 
 [[u]]= \esssup_{0<t<T} \|u(t)\|_{L_{\gamma_1}^2(U)}+\| \nabla u\|_{L^m_F(Q_T)}.
\eeqs 

Noting that  that $q<m$, we apply H\"{o}lder's inequality with powers  $\frac{m}q$ and $\frac{m}{m-q}$ to functions $F(x,t)^\frac{q}m|\nabla u|^{q}$ and $\gamma_2(x)F(x,t)^{-\frac{q}m}\chi_{{\rm supp}u}$, and obtain 
\begin{align*}
\int_U |\nabla u(t)|^{q}\gamma_2(x)dx
&\le \left(\int_U F(x,t)|\nabla u(t)|^{m}dx\right)^\frac{q}{m}  \left(\int_U \gamma_2(x)^{\frac{m}{m-q}}F(x,t)^{-\frac{q}{m-q}}\chi_{{\rm supp}u}(x,t)dx\right)^\frac{m-q}{m}\\
&\le \left(\int_U F(x,t)|\nabla u(t)|^{m}dx\right)^\frac{q}{m}  \mathcal{Z}^\frac{m-q}{m} 
\end{align*}
for almost all $t\in(0,T)$.
By definition of $p$, we have $2< p<  r$. Let $\alpha$ and $\beta$ be defined as in \eqref{betacond}. Then combining the preceding inequality with \eqref{estwphi}, we have
\beqs
\| u\|^{p}_{L_{\gamma_1}^{p}(Q_T)}
 \le  c_0^{\beta p}  [[u]]^{\alpha p}  \mathcal{Z}^\frac{\beta p(m-q)}{mq} \int_0^T\Big(\int_U F(x,t)|\nabla u|^{m}dx\Big)^\frac{\beta p}{m}dt 
 .
\eeqs

Note that $\beta p/m=1$, then
\beqs
\| u\|^{p}_{L_{\gamma_1}^{p}(Q_T)}
 \le  c_0^{\beta p} \mathcal{Z}^\frac{m-q}{q} [[u]]^{\alpha p}\left(\int_0^T\int_U F(x,t)|\nabla u|^{m}dxdt\right)^\frac{\beta p}{m}.
\eeqs

Since $\left(\int_0^T\int_U F(x,t)|\nabla u|^{m}dxdt\right)^\frac 1m\le [[u]]$, it follows that 
\beqs
\| u\|^{p}_{L_{\gamma_1}^{p}(Q_T)}
 \le c_0^{\beta p} \mathcal{Z}^\frac{m-q}{q} [[u]]^{\alpha p} [[u]]^{\beta p} 
 =  c_0^{m} \mathcal{Z}^\frac{m-q}{q} [[u]]^{ p} .
\eeqs

Taking power $1/p$ both sides of this inequality yields \eqref{gammaandK}.
\end{proof}

The preceding inequalities will be used with specific weights arising from the coefficients of the Forchheimer equations. We review them here.

The following exponent will be used throughout in our calculations
\beq\label{eq9}
a=\frac{\alpha_N}{\alpha_N+1}\in (0,1).
\eeq

We recall some  properties of $K(x,\xi)$.
We have from Lemmas III.5 and III.9 in \cite{ABHI1} that
\beq\label{Kderiva}
-aK(x,\xi)\le \xi \frac{\partial K(x,\xi)}{\partial\xi} \le 0\quad \forall \xi\ge 0.
\eeq
This implies  $K(x,\xi)$ is decreasing in $\xi$, hence 
\beq\label{decK}
K(x,\xi)\le K(x,0)=\frac 1{g(x,0)}=\frac 1{a_0(x)}.
\eeq

Define the main weight functions
\beqs
M(x)=\max\{ a_j(x):j=0,\ldots, N\},\quad
m(x)=\min\{a_0(x),a_N(x)\},
\eeqs
\beq\label{W12}
W_1(x)=\frac{a_N(x)^a}{2 N M(x)},\quad W_2(x)= \frac{NM(x)}{m(x)a_N(x)^{1-a}}.
\eeq

Note that
\beq\label{W1a}
W_1(x)a_N(x)^{2-a}=\frac{a_N(x)^2}{2NM(x)}\le \frac{a_N(x)}{2N}\le \frac{a_N(x)}2.
\eeq

From Lemma 1.1 of  \cite{CH1}, we have for all $\xi\ge 0$ that
\beq \label{KK}
\frac{2W_1(x)}{\xi^a +a_N(x)^a}\le K(x,\xi)\le \frac{W_2(x)}{\xi^a}
\eeq
and, consequently, 
\beq \label{Kwithsquare}
W_1(x)\xi^{2-a}-\frac{a_N(x)}{2} \le K(x,\xi)\xi^2\le W_2(x)\xi^{2-a}.
\eeq



For any number $r$ in $(1,\infty)$, we denote its conjugate exponent by $r'=r/(r-1)$.

We rewrite  Lemma \ref{PSob2} for our particular problem with specific weights.

\begin{corollary}\label{condAgamm}
Let function $K(x,\xi)$ and number $a\in(0,1)$ be defined by \eqref{eq4} and \eqref{eq9}, respectively.
Let $\varphi(x)$ be any positive function on $U$ and $W_1(x)$ be defined in \eqref{W12}.

Assume there are $r>2$ and $c_0>0$ such that
\beqs
\|u\|_{L^r_{\varphi}(U)}\le c_0 \|\nabla u\|_{L^{2-a}_{W_1}(U)}
\eeqs
for any $u(x)$ belonging to a space $\mathring X^{r,2-a}_{\varphi,W_1}(U)$.

Then for any $T>0$,  $u(x,t)\in \mathring X^{r,2-a}_{\varphi,W_1}(Q_T)$ and function $f(x,t)\ge 0$ on $Q_T$, one has
\beq\label{gamandK2}
\begin{aligned}
\| u\|_{L_\varphi^{4/r'}(Q_T)} 
&\le c_0^\frac{r'}2 \Big ( \int_U a_N(x)dx  + \esssup_{0<t<T} \int_U W_1(x)f(x,t)^{2-a}\chi_{{\rm supp}u}(x,t)dx \Big)^{\frac{ar'}{4(2-a)}}\\
&\quad \cdot  \Big\{\esssup_{0<t<T} \|u(t)\|_{L_\varphi^2(U)}+\Big(\int_0^T\int_U K(x,f(x,t)) |\nabla u(x,t)|^2dxdt\Big)^\frac{1}{2}\Big\}.
\end{aligned}
\eeq
\end{corollary}
\begin{proof}
Denote $\chi=\chi_{{\rm supp}u}$ and 
\beq 
 [[u]] = \esssup_{0<t<T} \|u(t)\|_{L_\varphi^2(U)}+\left(\int_0^T\int_U K(x,f(x,t)) |\nabla u(x,t)|^2dxdt\right)^\frac{1}{2}.
\eeq

Let $m=2$, $q=2-a$, $\gamma_1(x)=\varphi(x)$ and $\gamma_2(x)=W_1(x)$. Then two numbers $r$ and $q$ satisfy \eqref{qr}. The number  $p$ defined by \eqref{pdef} is
\beqs
p=4\left(1-\frac{1}{ r}\right)=\frac4{ r'}.
\eeqs

Let $F(x,t)=K(x,f(x,t))$. By Lemma \ref{PSob2}, we have following particular version of \eqref{gammaandK}
%
%
\beq\label{i1}
\| u\|_{L_\varphi^{p}(Q_T)}
\le  c_0^\frac{r'} 2   [[u]] \esssup_{0<t<T} \left(\int_U W_1(x)^{\frac{2}{a}}K(x,f(x,t))^{-\frac{2-a}{a}}\chi(x,t)dx\right)^\frac{ar'}{4(2-a)}.
\eeq 
For the last integral using \eqref{KK} and \eqref{ee3}, we have 
\begin{align*}
 W_1(x)^{\frac{2}{a}}K(x,f(x,t))^{-\frac{2-a}{a}}
&\le  W_1(x)^{\frac{2}{a}} \left[\frac{f(x,t)^a+a_N(x)^a}{2W_1(x)}\right]^{\frac{2-a}{a}}\\
& \le   W_1(x)a_N(x)^{2-a}+ W_1(x)f(x,t)^{2-a}.
\end{align*}
By \eqref{W1a}, we then have 
\begin{align*}
\int_U W_1(x)^{\frac{2}{a}}K(x,f(x,t))^{-\frac{2-a}{a}}\chi(x,t)dx
\le  \int_U a_N(x)dx+\int_U W_1(x)f(x,t)^{2-a}\chi(x,t)dx.
\end{align*}
Hence it follows \eqref{i1} that
\beqs
\| u\|_{L_\varphi^{p}(Q_T)}
 \le c_0^\frac{r'} 2   [[u]]
 \esssup_{0<t<T} \left( \int_U a_N(x)dx+\int_UW_1(x)f(x,t)^{2-a}\chi(x,t)dx\right)^\frac{ar'}{4(2-a)}.
\eeqs 
Thus  we obtain \eqref{gamandK2}.
\end{proof}

The following is a generalization of the convergence of fast decay geometry sequences in Lemma 5.6, Chapter II of \cite{LadyParaBook68}. It will be used in our version of  De Giorgi's  iteration.

\begin{lemma}[cf. \cite{HKP1}, Lemma A.2] \label{multiseq} 
Let $\{Y_i\}_{i=0}^\infty$ be a sequence of non-negative numbers satisfying
\beqs 
Y_{i+1}\le \sum_{k=1}^m A_k B^i  Y_i^{1+\mu_k}, \quad 
i =0,1,2,\cdots,
\eeqs
where  $B>1$, $A_k>0$  and $\mu_k>0$ for $k=1,2,\ldots,m$. Let $\mu=\min\{ \mu_k: 1\le k \le m \}.$

If $Y_0\le \min\big\{ (m^{-1} A_k^{-1} B^{-\frac 1 {\mu}})^{1/\mu_k} : 1\le k\le  m\big\}$
then $\lim_{i\to\infty} Y_i=0$.
\end{lemma}

The following simple property will help simplify large time estimates.

\begin{lemma}[c.f. \cite{CH1}, Lemma A.4]\label{Btt}
 Let $f(t)\ge 0$ be a $C^1$-function on $(0,\infty)$. Assume
\beqs
\beta=\limsup_{t\to\infty} [f'(t)]^-<\infty.
\eeqs
Then there is $T>0$ such that for any $t_2>t_1>T$,
\beqs
f(t_1)\le f(t_2)+(t_2-t_1)(\beta+1).
\eeqs
\end{lemma}


\myclearpage
\section{Reviews}\label{review}

In this section we review previous estimates obtained in \cite{CH1} for a solution $p(x,t)$ of the  IBVP \eqref{ppb}. They will be needed in  sections \ref{psec} and \ref{ptsec}.

Let $\Psi(x,t)$ be an extension  $\psi(x,t)$ from  $\Gamma\times (0,\infty)$  to $\bar U\times[0,\infty)$. Here, all estimates are stated in terms of $\Psi$, but can certainly be 
re-written in terms $\psi$, see e.g. \cite{HI1}.

Define $\bar{p}(x,t)=p(x,t)-\Psi(x,t)$.
Throughout the paper, we derive estimates for $\bar p$. The estimates for $p$ are easily obtained by using the triangle inequality  $|p|\le |\bar p|+|\Psi|$.

We assume:

\medskip

\textbf{(H1)} There is $c_1>0$ such that if $u(x)$ vanishes on $\Gamma$ then
\beq\label{P1}
\|u\|_{L^2_{\phi}(U)}\le c_1 \|\nabla u\|_{L^{2-a}_{W_1}(U)}.
\eeq

For the validity of \eqref{P1}, see Example \ref{SDCrmk}(b) with $\gamma_1=\phi$ and $\gamma_2=W_1$.

Here afterward, notation $\|\cdot\|_{L^p_f}$ stands for $ \|\cdot\|_{L^p_f(U)}$, and  $p_t$ means $\frac{\partial p}{\partial t}$.

Let
\beqs
B_1=\int_U a_N(x) dx,\quad B_*=\max\{B_1,1\},
\eeqs
and for $t\ge 0$,
\beqs
 G(t)= B_*+\|\nabla \Psi(t)\|_{L^2_{1/a_0}}^2+ \|\nabla \Psi(t)\|_{L^{2-a}_{W_1}}^{2-a}+\|\Psi_t(t)\|_{L^2_\phi}^\frac {2-a}{1-a},\quad
 G_1(t)= \|\nabla \Psi_t(t)\|_{L^2_{1/a_0}}^2.
\eeqs

Let ${\mathcal M}(t)$ be  a continuous function on $[0,\infty)$ that satisfies $\mathcal M(t)$  is increasing and $\mathcal M(t)\ge G(t)$ for all $t\ge 0$.
Denote
\beqs
\mathcal A=\limsup_{t\to\infty}G(t) \quad\text{and}\quad 
\mathcal B =\limsup_{t\to\infty}[G'(t)]^-.
\eeqs

Note that 
$\mathcal A,{\mathcal M}(t)\ge 1,B_1$ for all $t\ge 0$.

In the remainder of this section, the symbol $C$  denotes a generic positive constant which may change its values  from place to place, depends on 
number $a$ in \eqref{eq9} and the Sobolev constant $c_1$ in \eqref{P1}, but not on   individual functions $\phi(x)$, $a_i(x)$'s, the initial and boundary data.

\begin{theorem}[c.f. \cite{CH1}, Theorem 2.2]\label{thm32}
If $t>0$ then 
\beq\label{EstLtwo}
\int_U 
\bar{p}^2(x,t) \phi(x)dx \le \int_U \bar{p}^2(x,0)\phi(x) dx+ C{\mathcal M}(t)^{\frac 2{2-a}}.
\eeq

If $\mathcal A<\infty$ then 
\beq \label{noinitial}
\limsup_{t\to\infty} \int_U \bar{p}^2(x,t) \phi(x)  dx \le C \mathcal A^{\frac 2{2-a}}.
\eeq

If $\mathcal B<\infty$ then there is $T>0$ such that for all $t>T$ 
\beq\label{EstLtwo7}
\int_U 
\bar{p}^2(x,t) \phi(x)dx \le C(\mathcal B^{\frac 1{1-a}}+ G(t)^{\frac 2{2-a}}).
\eeq
\end{theorem}

Next, we recall weighted norm estimates for the pressure's derivatives.
The  differential inequality (3.6) from \cite{CH1} reads
\begin{multline}\label{ptrowest}
\frac d{dt}\Big(\int_U \bar{p}^2\phi dx + \int_U H(x,|\nabla p(x,t)|)dx \Big)
+ \int_U \bar p_t^2 \phi dx  + \frac14\int_U H(x,|\nabla p(x,t)|)dx\\
\le   C (G(t)+G_1(t)).
\end{multline}
Also, we have  an inequality of uniform Gronwall-type from \cite[Lemma 3.2]{CH1} for $t\ge 1$,
\begin{multline}\label{Gtwo}
\int_U H(x,|\nabla p(x,t)|)dx+\frac{1}{2}\int_{t-\frac 12}^t\int_U \bar{p}_t^2(x,\tau)\phi(x) dxd\tau\\
\le C\Big(\int_U \bar{p}^2(x,t-1)\phi(x) dx+\int_{t-1}^t (G(\tau)+G_1(\tau))d\tau\Big).
\end{multline}


\begin{theorem}[c.f. \cite{CH1}, Corollary  3.5]\label{cor35}
For $t>0$,
\begin{multline}\label{gradpW1}
\int_U W_1(x)|\nabla p(x,t)|^{2-a} dx \le e^{-\frac14t} \int_U H(x,|\nabla p(x,0)|)dx\\
+ C\Big(\int_U \bar{p}^2(x,0)\phi(x)  dx+\mathcal M^\frac2{2-a}(t) +\int_0^t e^{-\frac14(t-\tau)}G_1(\tau)  d\tau\Big).
\end{multline}

For $t\ge 1$,
\beq\label{C1}
 \int_U W_1(x)|\nabla p(x,t)|^{2-a} dx \le C\Big(\int_U \bar{p}^2(x,0)\phi(x) dx+{\mathcal M}(t)^\frac{2}{2-a} + \int_{t-1}^t G_1(\tau)d\tau\Big).
\eeq

If $\mathcal A<\infty$ then 
\begin{equation}\label{C3}
\begin{aligned}
 &\limsup_{t\to\infty}\int_U W_1(x)|\nabla p(x,t)|^{2-a} dx  \le C\Big(\mathcal A^\frac 2{2-a} + \limsup_{t\to\infty}\int_{t-1}^t  G_1(\tau)d\tau\Big).
\end{aligned}
\end{equation}

If $\mathcal B<\infty$ then there is $T>1$ such that  for all $t>T$,
\beq\label{C4}
\int_U W_1(x)|\nabla p(x,t)|^{2-a} dx
\le  C\Big(\mathcal B^{\frac 1{1-a}}+ G(t)^{\frac 2{2-a}}+\int_{t-1}^t G_1(\tau)d\tau\Big).
\eeq
\end{theorem}


\myclearpage
\section{Maximum estimates for the pressure}\label{psec}

We derive $L^\infty$-estimates for  the solution $p(x,t)$ of problem  \eqref{ppb}.
Let $p(x,t)$ and $\Psi(x,t)$ be the same as in section \ref{review}.
Let $\bar{p}(x,t)=p(x,t)-\Psi(x,t)$. Then  we have 
\begin{align}\label{pbareqn}
\phi(x) 
\begin{displaystyle}
 \frac{\partial\bar p}{\partial t}\end{displaystyle}
&= \nabla \cdot (K(x,|\nabla p|)\nabla p)-\phi(x)\Psi_t \quad \text{on }  U\times(0, \infty),\\
\bar{p}&=0\quad\text{on }\Gamma \times (0, \infty).\nonumber
\end{align}

We will make use of the parabolic Poincar\'e-Sobolev inequality \eqref{PSi2}.
Hence, we assume in this section that

\textbf{(H2)} Function $\phi(x)$ belongs to $L^1(U)$, and there are $r>2$ and $c_2>0$  such that
\beq\label{P2}
\|u\|_{L^r_{\phi}(U)}\le c_2 \|\nabla u\|_{L^{2-a}_{W_1}(U)}
\eeq
 for  functions $u(x)$ that vanish on the boundary $\Gamma$.

\medskip 

We have the following remarks on (H2):
 \begin{enumerate}
   \item[(a)] If $\phi(x)$ is the physical porosity function in applications, then $\phi(x)\le 1$, so it belongs to $L^1(U)$.
  \item[(b)] According to Example \ref{SDCrmk}(c), the number $r$ exists and inequality \eqref{P2} holds under (SDC) and condition \eqref{gams2} with $\gamma_1=\phi$ and $\gamma_2=W_1$.
  \item[(c)] Since $\phi\in L^1(U)$ and $r>2$, then, by H\"older's inequality, (H2) implies (H1) and \eqref{P1} holds with
  \beq\label{newc1}
   c_1=c_2\Big(\int_U \phi(x)dx\Big)^{\frac12-\frac1r}.
  \eeq
 \end{enumerate}

\medskip 

Here afterward, we fix $r$ in (H2) and constant $c_2$ in \eqref{P2}.
Note that $r'<2$. 

Denote by  $r_0$  the number $p$  defined by \eqref{pcond} with $q=2-a$, that is, 
$$r_0=2+(2-a)\big(1-\frac2{r}\big)> 2.$$

The following estimates use a fixed parameter $r_1$, which is a  number in interval $(1, r_0/2)$. 

\begin{proposition}\label{Linftee}
For any $T_0 \ge 0$, $T>0$ and  $\theta\in(0,1)$, one has 
\begin{multline}\label{inftestDG}
\|\bar{p}(t)\|_{L^\infty(U\times (T_0+\theta T,T_0+T))}
\le \bar C\max\{1,c_2\}^\frac{2-a}{r_0-2} \big[(\theta T)^{-\frac 12}+(\theta T)^{-\frac{1}{2-a}}\big]^{\kappa_1} (1+ \omega_{T_0,T})^{\kappa_2}\\
\cdot  \Big(  \|\bar{p}\|^{\nu_1}_{L_\phi^2(U\times (T_0,T_0+T))}+ \|\bar{p}\|_{L_\phi^2(U\times (T_0,T_0+T))}^{\nu_2}\Big),
\end{multline}
where constant  $\bar C>0$ is independent of $c_2$, $T_0$, $T$, and $\theta$, 
\beqs 
\kappa_1=\frac{r_0}{r_0-2},\ {\kappa_2}={\frac{r_0(r_1-1)}{2r_0+(r_0-2)r_1(2-a)}}, \ 
{\nu_1}=\frac{r_0-2r_1}{r_0+(r_0-2) r_1},\ 
{\nu_2}=\frac{2(r_0-2+a)}{(2-a)(r_0-2)}, 
\eeqs
and
\begin{multline}\label{oTT}
\omega_{T_0,T}
= T\int_U a_N(x)^{ r'_1}\phi(x)^{1- r'_1}dx + T^{ r'_1}\int_{T_0}^{T_0+T}\int_U |\Psi_t(x,t)|^{2 r'_1} \phi(x) dxdt\\
+\int_{T_0}^{T_0+T}\int_U \big(W_1(x)|\nabla \Psi(x,t)|^{2-a}+a_0(x)^{-1} |\nabla \Psi(x,t)|^{2}\big)^{ r'_1}\phi(x)^{1- r'_1} dxdt.
\end{multline}
\end{proposition}


\begin{proof}
We use De Giorgi's iteration, see  \cite{LadyParaBook68,DiDegenerateBook}. Without loss of generality we assume $T_0=0$ and $\|\bar{p}\|_{L_\phi^2(U\times (0,T))}>0$. 
In the following calculations, generic number $\bar C>0$ and specific constants $\bar C_1,\bar C_2>0$ depend on numbers $a$, $r$ and $r_1$,  but not on $c_2$ in \eqref{P2}.

\textbf{Step 1.} For each $k\ge 0$, let $\bar{p}^{(k)}=\max\{\bar{p}-k, 0\}$. 
Note that $\bar p^{(k)}=0$ on $\Gamma$.

Let $\chi_k(x,t)$  denote the characteristic of the set $\{(x,t)\in Q_T: \bar{p}^{(k)}(x,t)> 0\}$.

Let $\zeta=\zeta(t)\ge 0$ be a smooth function on $\mathbb R$ with  $\zeta(t)=0$ for $t\le 0$.

Multiplying equation \eqref{pbareqn} by $\bar{p}^{(k)}\zeta$, then integrating over the domain $U$, and using integration by parts, we get 
\begin{align*}
\frac 12  \int_U \frac {\partial|\bar{p}^{(k)}|^2}{\partial t}\zeta \phi dx= -\int_U   K(x,|\nabla {p}|)\nabla p\cdot \nabla  \bar{p}^{(k)} \zeta dx- \int_U \bar{p}^{(k)} \zeta \Psi_t \phi dx.
\end{align*}
For the first integral on the right-hand side, we write
\begin{align*}
\nabla p\cdot \nabla  \bar{p}^{(k)} 
&=\nabla \bar p\cdot \nabla  \bar{p}^{(k)} +\nabla \Psi\cdot  \nabla \bar{p}^{(k)}
=|\nabla \bar{p}^{(k)}|^2 +\chi_k \nabla \Psi\cdot \nabla  \bar{p}^{(k)}\\
&\ge |\nabla \bar{p}^{(k)}|^2 - \frac12 (|\nabla \bar{p}^{(k)}|^2+\chi_k |\nabla \Psi|^2)
=\frac12 |\nabla \bar{p}^{(k)}|^2 - \frac12\chi_k |\nabla \Psi|^2.
\end{align*}

Hence,  we gain
\begin{align*}
\frac 12  \int_U \frac {\partial|\bar{p}^{(k)}|^2}{\partial t}\zeta \phi dx
&\le - \frac 12\int_U K(x,|\nabla {p}|)|\nabla \bar{p}^{(k)} |^2\zeta dx\\
&\quad+ \frac 12\int_U K(x,|\nabla {p}|)|\nabla \Psi|^2\chi_k \zeta dx+\int_U \bar{p}^{(k)}|\Psi_t| \zeta \phi dx.
\end{align*}

Let $\varepsilon>0$. Using \eqref{decK} to bound $K(x,|\nabla p|)$ in the middle integral on the right-hand side, and applying Cauchy's inequality to the last integral yield
\begin{multline}\label{afC}
\int_U  \frac {\partial|\bar{p}^{(k)}|^2}{\partial t}\zeta \phi dx+ \int_U K(x,|\nabla {p}|)|\nabla \bar{p}^{(k)}|^2 \zeta dx\\
 \le \varepsilon\int_U |\bar{p}^{(k)}|^2\zeta \phi dx+ \varepsilon^{-1}\int_U \chi_k\cdot|\Psi_t|^2\zeta \phi dx+\int_U \chi_k \cdot a_0(x)^{-1}|\nabla \Psi|^2\zeta dx.
\end{multline}

For the second integral on the left-hand side, using relation \eqref{KK} and triangle inequality,  we have 
\begin{align*}
K(x,|\nabla {p}|)|\nabla \bar{p}^{(k)}|^2 
&\ge \frac{2W_1(x)}{|\nabla {p}|^a+a_N(x)^a}|\nabla \bar{p}^{(k)}|^2 \\
&\ge  \frac{2W_1(x)|\nabla \bar{p}^{(k)}|^2}{(|\nabla \bar{p}|+|\nabla \Psi|)^a+a_N(x)^a}
= \frac{2W_1(x)|\nabla \bar{p}^{(k)}|^2}{(|\nabla \bar{p}^{(k)}|+|\nabla \Psi|)^a+a_N(x)^a}.
\end{align*}

Applying  inequality \eqref{ee7} to $|\nabla \bar{p}^{(k)}|^2$ in the numerator gives
\begin{align*}
K(x,|\nabla {p}|)|\nabla \bar{p}^{(k)}|^2 
&\ge \frac{W_1(x)(|\nabla \bar{p}^{(k)}|+|\nabla \Psi|)^2}{(|\nabla \bar{p}^{(k)}|+|\nabla \Psi|)^a+a_N(x)^a}-\frac{2W_1(x)|\nabla \Psi|^2}{(|\nabla \bar{p}^{(k)}|+|\nabla \Psi|)^a+a_N(x)^a}\\
&\ge \frac{W_1(x)(|\nabla \bar{p}^{(k)}|+|\nabla \Psi|)^2}{(|\nabla \bar{p}^{(k)}|+|\nabla \Psi|)^a+a_N(x)^a}-2W_1(x)|\nabla \Psi|^{2-a}.
\end{align*}

To bound the first term on the right-hand side, we use the following inequality. For $b>0$ and $\xi\ge 0$, by considering two cases $\xi<b$ and $\xi\ge b$, one can easily prove that 
\begin{align*}
\frac{\xi^2}{\xi^a+b^a} \ge \frac 12 (\xi^{2-a}-b^{2-a}).
\end{align*}

Applying this inequality to $\xi=|\nabla \bar{p}^{(k)}|+|\nabla \Psi|$ and $b=a_N(x)$, we obtain
\begin{align*}
K(x,|\nabla {p}|)|\nabla \bar{p}^{(k)}|^2
&\ge \frac{W_1(x)}2\Big [(|\nabla \bar{p}^{(k)}|+|\nabla \Psi|)^{2-a}-a_N(x)^{2-a}\Big]-2W_1(x)|\nabla \Psi|^{2-a}\\
&\ge \frac12 W_1(x)|\nabla \bar{p}^{(k)}|^{2-a}-\frac12 W_1(x) a_N(x)^{2-a}-2W_1(x)|\nabla \Psi|^{2-a}.
\end{align*}
Using \eqref{W1a} to bound $W_1(x) a_N(x)^{2-a}$  we have 
\beq\label{Kmaxe}
\begin{aligned}
K(x,|\nabla {p}|)|\nabla \bar{p}^{(k)}|^2
&\ge \frac12 W_1(x) |\nabla \bar{p}^{(k)}|^{2-a}-\frac14 a_N(x)-2W_1(x)|\nabla \Psi|^{2-a}.
\end{aligned}
\eeq
In \eqref{afC}, utilizing \eqref{Kmaxe} and using the  product rule of derivation for the first term on the left-hand side, we have
\begin{align*}
& \frac d{dt} \int_U |\bar{p}^{(k)}|^2\zeta \phi  dx+\frac 12 \int_U  W_1(x) |\nabla \bar{p}^{(k)}|^{2-a} \zeta dx
\le\varepsilon\int_U |\bar{p}^{(k)}|^2\zeta \phi dx+\int_U |\bar{p}^{(k)}|^2 |\zeta_t| \phi dx\\
&\quad + \int_U \chi_k\cdot \Big[ \frac 14a_N(x)+2W_1(x)|\nabla \Psi|^{2-a}+ \varepsilon^{-1}|\Psi_t|^2 \phi + a_0(x)^{-1}|\nabla \Psi|^2\Big]\zeta dx.
\end{align*}
Then integrating in $t$, using the fact that $\zeta(0)=0$ and taking supremum on $(0,T)$, we have 
\begin{align*}
&\sup_{0<t<T}\int_U |\bar{p}^{(k)}(x,t)|^2\zeta(t) \phi(x) dx+ \frac 12\int_0^T\int_U  W_1(x) |\nabla \bar{p}^{(k)}|^{2-a} \zeta dxdt\\
&\le 2\varepsilon T \sup_{0<t<T}\int_U |\bar{p}^{(k)}(x,t)|^2\zeta(t) \phi(x)  dx+ 2\int_0^T\int_U |\bar{p}^{(k)}|^2 |\zeta_t| \phi dxdt\\
&\quad+ 2\int_0^T\int_U \chi_k\cdot \Big[ \frac 14a_N(x)+2W_1(x)|\nabla \Psi|^{2-a}+ \varepsilon^{-1}|\Psi_t|^2 \phi + a_0(x)^{-1}|\nabla \Psi|^2\Big]\zeta dxdt.
\end{align*}
Choosing $\varepsilon=\frac 1{4T}$, and absorbing the first term on the right-hand side into the left yield
\begin{multline}\label{F0}
\sup_{0<t<T}\int_U |\bar{p}^{(k)}(x,t)|^2\zeta(t) \phi(x) dx+ \int_0^T\int_U W_1(x)|\nabla {\bar p}^{(k)}|^{2-a} \zeta dxdt\\
\le 4\int_0^T\int_U |\bar{p}^{(k)}|^2 |\zeta_t| \phi dxdt
+16\int_0^T\int_U \chi_k\cdot \mathcal E(x,t)\zeta dxdt,
\end{multline}
where
\beq
\mathcal E(x,t)= a_N(x)+W_1(x)|\nabla \Psi|^{2-a}+ T|\Psi_t|^2 \phi + a_0(x)^{-1}|\nabla \Psi|^2.
\eeq
\textbf{Step 2.}
We will iterate \eqref{F0} with different values of $k$ and different functions $\zeta$. 
Let $i\ge 0$ be any integer.
Denote $t_i=\theta T\left(1-\frac 1{2^i}\right)$. Then  $t_0=0<t_1<t_2<...<\theta T$ and $t_i\to \theta T$ as $i\to\infty$.

Let $\zeta(t)$ be a smooth function from $\mathbb R$ to $[0,1]$ such that
$$\zeta_i(t)=\left\{ \begin{array}{rcl} 0 & \mbox{for} & t \le t_i \\ 1 & \mbox{for} & t \ge t_{i+1} \end{array}\right.
\quad\text{and}\quad 0\le \zeta'_i(t)\le \frac 2{t_{i+1}-t_i}=\frac {2^{i+2}}{\theta T}\quad\forall t\in\mathbb R.$$

Let $M_0$ be a fixed positive number that will be determined later.

Define $k_i=M_0(1-2^{-i})$
and the set $A_{i,j}=\{ (x,t): p(x,t)>k_i, t \in(t_j,T) \}$ for $i,j\ge 0$. 

Applying inequality \eqref{F0} to $k=k_{i+1}$ and $\zeta=\zeta_i\le 1$ gives 
%
\begin{align*}
F_i
&\eqdef \sup_{0<t<T}\int_U |\bar{p}^{(k_{i+1})}(x,t)|^2\zeta_i(t) \phi(x) dx+ \int_0^T\int_U W_1(x)|\nabla {\bar p}^{(k_{i+1})}(x,t)|^{2-a} \zeta_i(t) dxdt\\
&\le \bar C\int_0^T\int_U |\bar{p}^{(k_{i+1})}|^2 \zeta_i' \phi dxdt+\bar C\int_0^T\int_U \chi_{k_{i+1}}\mathcal E\zeta_idxdt.
\end{align*}

On the right-hand side, using the properties of $\zeta_i$ we bound 
\beqs
F_i \le \bar C\frac{2^{i}}{\theta T}\int_{t_i}^T\int_U |\bar{p}^{(k_{i+1})}|^2 \phi dxdt
+\bar C\int_{t_{i}}^T\int_U  \chi_{k_{i+1}}\mathcal Edxdt.
\eeqs
Applying  H\"older's inequality with powers $ r_1$ and $ r'_1$ to the last double integral yields
\begin{align*}
F_i 
\le \bar C\frac{2^{i}}{\theta T} \int_{t_i}^T \int_U |\bar{p}^{(k_{i+1})}|^2\phi dxdt+\bar C\left(\int_{t_i}^T\int_U \chi_{k_{i+1}}^{ r_1}\phi dxdt\right)^{\frac 1{ r_1}}
\left(\int_{t_i}^T\int_U \mathcal E^{ r'_1}\phi^{1- r'_1}dxdt\right)^{\frac 1{ r'_1}}.
\end{align*}

Denote $\omega_T=\omega_{0,T}$ as defined in \eqref{oTT}.
Then
\beq
\int_{t_i}^T\int_U \mathcal E^{ r'_1} \phi^{1- r'_1}dxdt
\le \bar C \omega_T.
\eeq
Then
\begin{align*}
F_i &\le \bar C\frac{2^{i}}{\theta T} \int_{t_i}^T \int_U |\bar{p}^{(k_{i+1})}|^2\phi dxdt+\bar C\omega_T^\frac{1}{ r'_1}\left(\int_{t_i}^T\int_U \chi_{k_{i+1}}^{ r_1}\phi dxdt\right)^{\frac 1{ r_1}}\\
&\le \bar C\frac{2^{i}}{\theta T} \| \bar{p}^{(k_{i+1})} \|^2_{L_\phi^2(A_{i+1,i})}+\bar C\omega_T^{\frac 1{ r'_1}}\bar \mu(A_{i+1,i})^{\frac 1{ r_1}},
\end{align*}
where $\bar \mu=\bar \mu_\phi$ is the measure defined in \eqref{meas2} with $f(x,t)=\phi(x)$.
Since $A_{i+1,i}\subset A_{i,i}$ and $\bar p^{(k)}$ is decreasing in $k$, we derive
\beq \label{Fineq0}
F_i 
\le \bar C\frac{2^{i}}{\theta T} \| \bar{p}^{(k_i)} \|^2_{L_\phi^2(A_{i,i})}+\bar C\omega_T^{\frac 1{ r'_1}}\bar \mu(A_{i+1,i})^{\frac 1{ r_1}}.
\eeq
We estimate the measure $\bar \mu(A_{i+1,i})$. Using $A_{i+1,i}\subset A_{i,i}$ again and definition of $\bar p^{(k)}$
 \beqs
\|\bar{p}^{(k_i)} \|_{L_\phi^2(A_{i,i})}^2\ge \|\bar{p}^{(k_i)} \|_{L_\phi^2(A_{i+1,i})}^2\ge (k_{i+1}-k_{i})^2\bar \mu(A_{i+1,i}).
\eeqs
This implies
\beq\label{Aii}
\bar \mu(A_{i+1,i}) \le (k_{i+1}-k_{i})^{-2} \|\bar{p}^{(k_i)} \|_{L_{\phi}^2(A_{i,i})}^2
=4^{i+1}M_0^{-2}\|\bar{p}^{(k_i)} \|_{L_{\phi}^2(A_{i,i})}^2.
\eeq
Then \eqref{Fineq0} yields

\beq\label{Fineq}
\begin{aligned}
F_i \le \bar C\frac{2^{i}}{\theta T} \| \bar{p}^{(k_{i})} \|^2_{L_\phi^2(A_{i,i})}+\bar C4^{\frac{i}{ r_1}}M_0^{-\frac{2}{ r_1}} \omega_T^{\frac 1{ r'_1}} \|\bar{p}^{(k_i)} \|_{L_\phi^2(A_{i,i})}^{\frac 2{ r_1}}.
\end{aligned}
\eeq

\textbf{Step 3.} Applying inequality \eqref{PSi2} of Lemma \ref{PSob1} to $r>2$, $q=2-a$, the weights $\gamma_1(x)=\phi(x)$, $\gamma_2(x)=W_1(x)$, and the function $u(x,t)=\bar p^{(k_{i+1})}(x,t)\zeta_i(t)$, we have
\begin{align*}
&\|\bar{p}^{(k_{i+1})}\zeta_i \|_{L_\phi^{r_0}(Q_T)}\\
&\le c_2^{\frac{2-a}{r_0}}\Big[ \esssup_{0<t<T}\|\bar{p}^{(k_{i+1})} (t)\zeta_i(t)\|_{L_\phi^2(U)}+\Big( \int_0^T \int_UW_1(x)|\nabla (\bar{p}^{(k_{i+1})}\zeta_i )|^{2-a} dxdt \Big)^\frac{1}{2-a} \Big]\\
&\le c_2^{\frac{2-a}{r_0}}\Big[ \sup_{0<t<T}\Big(\int_U |\bar{p}^{(k_{i+1})}(x,t)|^2\zeta_i(t) \phi(x) dx\Big)^\frac12+ \Big(\int_0^T\int_U W_1(x)|\nabla {\bar p}^{(k_{i+1})}|^{2-a} \zeta_i dxdt\Big)^\frac1{2-a} \Big].
\end{align*}

Above, we used the fact that $\zeta_i$ is a function of $t$ only, and $0\le \zeta_i\le 1$.
Therefore,
\beq\label{pii}
\|\bar{p}^{(k_{i+1})}\zeta_i \|_{L_\phi^{r_0}(Q_T)}
\le c_2^{\frac{2-a}{r_0}}(F_i^\frac12+F_i^\frac1{2-a}).
\eeq

Since $\zeta_i=1$ on $[t_{i+1},T]$ and $t_i\le t_{i+1}$, we have from  \eqref{pii} that 
\beqs
\|\bar{p}^{(k_{i+1})} \|_{L_\phi^{r_0}(A_{i+1,i+1})}
\le \|\bar{p}^{(k_{i+1})}\zeta_i \|_{L_\phi^{r_0}(Q_T)}
\le c_2^{\frac{2-a}{r_0}}(F_i^{\frac 12}+F_i^{\frac{1}{2-a}}).
\eeqs

By H\"older's inequality and by the fact that $A_{i+1,i+1}\subset A_{i+1,i}$ we have  
\begin{align*}
\|\bar{p}^{(k_{i+1})} \|_{L_\phi^{2}(A_{i+1,i+1})}
&\le \bar \mu(A_{i+1,i+1})^{\frac 12-\frac{1}{r_0}}\|\bar{p}^{(k_{i+1})} \|_{L_\phi^{r_0}(A_{i+1,i+1})}
\le c_2^{\frac{2-a}{r_0}} \bar \mu(A_{i+1,i})^{\frac 12-\frac{1}{r_0}} (F_i^{\frac 12}+F_i^{\frac{1}{2-a}}).
\end{align*}
Combining this with \eqref{Aii}, \eqref{Fineq}, and using inequality \eqref{ee2} yield 
\begin{align*}
\|\bar{p}^{(k_{i+1})} \|_{L_\phi^{2}(A_{i+1,i+1})}
&\le  \bar Cc_2^{\frac{2-a}{r_0}}(4^{i+1}M_0^{-2})^{\frac 12-\frac{1}{r_0}}\|\bar{p}^{(k_i)} \|_{L_\phi^2(A_{i,i})}^{1-\frac{2}{r_0}}\\
&\quad \cdot \Big\{\Big(\frac{2^{i}}{\theta T} \Big)^{\frac 12}\| \bar{p}^{(k_{i})} \|_{L_\phi^2(A_{i,i})}
+\Big(4^{\frac{i}{ r_1}}M_0^{-\frac{2}{ r_1}}\Big)^\frac 1{2} \omega_T^{\frac 1{2 r'_1}} \|\bar{p}^{(k_i)} \|_{L_\phi^2(A_{i,i})}^{\frac 1{ r_1}}\\
&\quad +\Big(\frac{2^{i}}{\theta T}\Big)^\frac 1{2-a} \| \bar{p}^{(k_{i})} \|^\frac 2{2-a}_{L_\phi^2(A_{i,i})}+\Big(4^{\frac{i}{ r_1}}M_0^{-\frac{2}{ r_1}}\Big)^\frac 1{2-a} \omega_T^{\frac 1{ r'_1(2-a)}} \|\bar{p}^{(k_i)} \|_{L_\phi^2(A_{i,i})}^{\frac 2{ r_1(2-a)}}\Big\}.
\end{align*}

For $i\ge 0$, define $Y_i=\|\bar{p}^{(k_i)}\|_{L_\phi^2(A_{i,i})}=\|\bar{p}^{(k_i)}\|_{L_\phi^2(U\times(t_i,T))}$.
We write the preceding inequality as 
\beq\label{Yi1}
Y_{i+1}\le 4^i \cdot\Big\{D_1Y_i^{2-\frac{2}{r_0}}+D_2Y_i^{\frac 1{ r_1}+1-\frac{2}{r_0}}+ D_3Y_i^{\frac 2{2-a}+1-\frac{2}{r_0}}+D_4Y_i^{\frac 2{ r_1(2-a)}+1-\frac{2}{r_0}}\Big\}
\eeq
for all $i\ge 0$, where 
\begin{align*}
D_1&=\bar C_1c_2^{\frac{2-a}{r_0}}M_0^{-\frac{r_0-2}{r_0}}(\theta T)^{-\frac 12},&
D_2&=\bar C_1c_2^{\frac{2-a}{r_0}}M_0^{-\frac{r_0-2}{r_0}-\frac{1}{ r_1}}\omega_T^{\frac 1{2 r'_1}},\\
D_3&=\bar C_1c_2^{\frac{2-a}{r_0}}M_0^{-\frac{r_0-2}{r_0}}(\theta T)^{-\frac 1{2-a}},& 
D_4&=\bar C_1c_2^{\frac{2-a}{r_0}}M_0^{-\frac{r_0-2}{r_0}-\frac{2}{ r_1(2-a)}}\omega_T^{\frac 1{ r'_1(2-a)}}
\end{align*}
with some $\bar C_1>0$.
Let 
$$e_1=1-\frac{2}{r_0},\quad  e_2=\frac 1{ r_1}-\frac{2}{r_0}, \quad  e_3=\frac 2{2-a}-\frac{2}{r_0}, \quad e_4=\frac 2{ r_1(2-a)}-\frac{2}{r_0}.$$
Then $e_1,e_2,e_3,e_4>0$ and \eqref{Yi1} becomes 
\begin{align*}
Y_{i+1}\le 4^i \Big(D_1Y_i^{1+e_1}+ D_2Y_i^{1+e_2}+D_3Y_i^{1+e_3}+ D_4Y_i^{1+e_4}\Big).
\end{align*}

\textbf{Step 4.} We apply  Lemma \ref{multiseq} with values $m=4$, $B=4>1$, $A_k=D_k>0$ and $\mu_k=e_k$, where $k=1,2,3,4$. 
If $M_0$ is chosen sufficiently large such that 
\beq\label{Y0cond}
Y_0\le \bar C_2\min\{D_k^{-1/e_k}:k=1,2,3,4\}
\eeq
with an appropriate positive $\bar C_2$, then $\lim_{i\to\infty}Y_i=0$.
 This gives 
\beq\label{pm0} 
\int_{\theta T}^T\int_U |\bar{p}^{(M_0)}|^2\phi(x)dx dt=0,
\eeq
which implies $\bar{p}^{(M_0)}(x,t)\phi(x)=0$ a.e. in $U\times(\theta T,T)$.
Since $\phi(x)>0$, we have $\bar{p}^{(M_0)}(x,t)=0$, or equivalently, $\bar{p}(x,t)\le M_0$ a.e. in $U\times(\theta T,T)$. Repeating the proof with $-p(x,t)$ replacing $p(x,t)$, and $-\Psi(x,t)$ replacing $\Psi(x,t)$, we obtain
\beq\label{pM}
|\bar{p}(x,t)|\le M_0 \quad \text{a.e. in } U\times(\theta T,T).
\eeq

It remains to find $M_0$ that satisfies \eqref{Y0cond}. Note that $k_0=0$ and $Y_0\le \|\bar{p}\|_{L_\phi^2(U\times (0,T))}$. It suffices to have
\begin{align}
\|\bar{p}\|_{L_\phi^2(U\times (0,T))} \le \bar C_2\min\{D_k^{-1/e_k}:k=1,2,3,4\},
\end{align}
which is equivalent to
\beq\label{ppow}
\begin{aligned}
M_0 &\ge \bar C c_2^\frac{2-a}{r_0-2} [(\theta T)^{-\frac 12}]^{\frac{r_0}{r_0-2}} \|\bar{p}\|^{\frac{e_1r_0}{r_0-2}}_{L_\phi^2(U\times (0,T))}, \\
 M_0&\ge  \bar C c_2^\frac{(2-a) r_1}{r_0+(r_0-2) r_1}  \omega_T^{\frac{r_0 r_1}{2 r'_1(r_0+(r_0-2) r_1)}} \|\bar{p}\|_{L_\phi^2(U\times (0,T))}^{\frac{e_2  r_1r_0}{r_0+(r_0-2) r_1}}, \\
 M_0 &\ge \bar C c_2^\frac{2-a}{r_0-2} [(\theta T)^{-\frac{1}{2-a}}]^{\frac{r_0}{r_0-2}} \|\bar{p}\|^{\frac{e_3r_0}{r_0-2}}_{L_\phi^2(U\times (0,T))}, \\
 M_0&\ge \bar C c_2^\frac{(2-a)^2 r_1}{2 r_0+( r_0-2) r_1(2-a)}  \omega_T^{\frac{ r_0 r_1}{ r'_1(2 r_0+( r_0-2) r_1(2-a))}} \|\bar{p}\|_{L_\phi^2(U\times (0,T))}^{\frac{e_4  r_1 r_0(2-a)}{2 r_0+( r_0-2) r_1(2-a)}}.
\end{aligned}
\eeq
We compare the lower bounds of $M_0$ in \eqref{ppow}.

For $\|\bar{p}\|_{L_\phi^2(U\times (0,T))}$-terms, we use inequality \eqref{ee4}, hence need to identify their maximum and minimum powers.
Note that  
$$e_3>e_1> e_2\quad \text{and} \quad e_3>e_4>e_2.$$
Then\beq\label{cp1} \frac{e_3 r_0}{ r_0-2}>\frac{e_1 r_0}{ r_0-2}>\frac{e_1  r_1 r_0}{ r_0+( r_0-2) r_1}>\frac{e_2  r_1 r_0}{ r_0+( r_0-2) r_1},
\eeq
\beq\label{cp2} 
\frac{e_3 r_0}{ r_0-2}>\frac{e_4 r_0}{ r_0-2}
>\frac{e_4 r_1 r_0(2-a)}{2 r_0+( r_0-2) r_1(2-a)}.
\eeq
Therefore, the maximum power of  $\|\bar{p}\|_{L_\phi^2(U\times (0,T))}$ in \eqref{ppow} is
\beqs
\frac{e_3 r_0}{ r_0-2}={\nu_2}.
\eeqs

To find the minimum power, we compare the smallest powers in \eqref{cp1} and \eqref{cp2}. With explicit calculations
\beqs
\frac{e_4 r_1 r_0(2-a)}{2 r_0+( r_0-2) r_1(2-a)}
=\frac{ r_0-(2-a) r_1}{ r_0+( r_0-2) r_1(2-a)/2}
>\frac{ r_0-2 r_1}{ r_0+( r_0-2) r_1}=\frac{e_2  r_1 r_0}{ r_0+( r_0-2) r_1}.
\eeqs
Therefore, the minimum   power of  $\|\bar{p}\|_{L_\phi^2(U\times (0,T))}$ in \eqref{ppow} is
\beqs
 \frac{e_2  r_1 r_0}{ r_0+( r_0-2) r_1}={\nu_1}.
\eeqs

For two $\omega_T$-terms in \eqref{ppow}, the powers of $\omega_T$ satisfy 
$$\qquad{\frac{ r_0 r_1}{2 r'_1( r_0+( r_0-2) r_1)}}\le {\frac{ r_0 r_1}{ r'_1(2 r_0+( r_0-2) r_1(2-a))}}={\kappa_2}.$$

For $(\theta T)$-terms, we just add $(\theta T)^{-\frac{1}{2-a}}$ to $(\theta T)^{-\frac 12}$.
Also, the maximum power of $c_2$ is $\frac{2-a}{ r_0-2}$.

In summary, each lower bound in \eqref{ppow} is less than or equal to
\beqs
\bar C\max\{1,c_2\}^\frac{2-a}{ r_0-2} (1+ \omega_T)^{\kappa_2}
\Big\{(\theta T)^{-\frac 12}+(\theta T)^{-\frac{1}{2-a}}\Big\}^{\frac{ r_0}{ r_0-2}} \Big\{  \|\bar{p}\|_{L_\phi^2(U\times (0,T))}^{\nu_1}+ \|\bar{p}\|_{L_\phi^2(U\times (0,T))}^{\nu_2}\Big\}.
\eeqs
Hence we choose this value as $M_0$, and  obtain \eqref{inftestDG} from \eqref{pM}.
The proof is complete.
\end{proof}

We combine Proposition \ref{Linftee} with estimates in section \ref{review} to give $L^\infty$-estimates in terms of the initial and boundary data.
Define for $t>s\ge 0$,
\begin{multline}
N_1(s,t)= \max\Big\{1,\int_U a_N(x)^{ r'_1}\phi(x)^{1- r'_1}dx\Big\}\\ +\int_s^t\int_U \Big[\Big(W_1(x)|\nabla \Psi(x,\tau)|^{2-a}+a_0(x)^{-1} |\nabla \Psi(x,\tau)|^{2}\Big)^{ r'_1}\phi(x)^{1- r'_1} +  |\Psi_t(x,\tau)|^{2 r'_1} \phi(x)\Big] dxd\tau.
\end{multline}

Then $N_1(s,t)\ge 1$ and $\omega_{T_0,T}$ in \eqref{oTT} satisfies
\beq\label{obo}
\omega_{T_0,T}\le 2\max\{1,T\}^{ r'_1} N_1(T_0,T_0+T).
\eeq

In the next theorem, we assume also (H1). The generic positive constant $C$ depends on  $a$, $r$, $r_1$,  $c_1$  in \eqref{P1}, and $c_2$ in \eqref{P2}, but not on  individual functions $\phi(x)$, $a_i(x)$'s,  the initial and boundary data.
\begin{theorem} \label{mainp}
Let the notation be the same as  in Proposition \ref{Linftee} and Theorem \ref{thm32}. 
\begin{enumerate}
\item[\rm (i)] If $t\in(0,1)$, then 
\beq\label{smallt}
\|\bar{p}\|_{L^\infty(U\times(t/2,t))}
\le Ct^{-\kappa_3} N_1(0,t)^{\kappa_2} 
 \Big(  \|\bar{p}(0)\|_{L_\phi^2(U)}+\mathcal M(t)^\frac1{2-a}\Big)^{\nu_2},
\eeq
where
\beq
\kappa_3=\frac{\kappa_1}{2-a}-\frac{\nu_1}2=\frac{r_0}{(2-a)(r_0-2)}-\frac{r_0-2r_1}{2(r_0+(r_0-2) r_1)}>0.
\eeq

If $t\ge 1$, then 
\beq\label{larget}
\|\bar{p}\|_{L^\infty(U\times(t-\frac12,t))}
\le C N_1(t-1,t)^{\kappa_2} \Big(   \|\bar{p}(0)\|_{L_\phi^2(U)} + \mathcal M(t)^\frac1{2-a}\Big)^{\nu_2}.
\eeq

\item[\rm (ii)] If $\mathcal A<\infty$ then
\beq\label{limpG}
\limsup_{t\to\infty}\|\bar{p}\|_{L^\infty(U\times(t-\frac12,t))}
\le C \big(\limsup_{t\to\infty} N_1(t-1,t)\big)^{\kappa_2}  \mathcal A^{\frac {\nu_2}{2-a}}.
\eeq

\item[\rm (iii)] If $\mathcal B<\infty$ then there is $T>0$ such that  for all $t>T$
\beq\label{forlim11}
\|\bar{p}\|_{L^\infty(U\times(t-\frac12,t))}
\le C N_1(t-1,t)^{{\kappa_2}}  \big(\mathcal B^{\frac 1{2(1-a)}}+ G(t)^{\frac {1}{2-a}}\big)^{\nu_2}.
\eeq
\end{enumerate}
\end{theorem}
\begin{proof}
By remark (c) after (H2), the condition (H1) in section \ref{review} is met with constant $c_1$ now specified by \eqref{newc1}. 
Thus, all constants $C$'s in estimates of section \ref{review} now depend on this $c_1$.

(i) Let  $t\in (0,1)$. Applying  inequality \eqref{inftestDG} to $T_0=0$, $T=t<1$ and $\theta=1/2$, and taking into account estimate \eqref{obo}, we have 
\begin{align*}
\|\bar{p}\|_{L^\infty(U\times(t/2,t))}
&\le C (t^{-\frac 12}+t^{-\frac{1}{2-a}})^{\kappa_1} N_1(0,t)^{\kappa_2} \Big(t^{\nu_1/2} \sup_{0\le \tau\le t} \|\bar{p}(\tau)\|^{\nu_1}_{L_\phi^2(U)}+ t^{\nu_2/2} \sup_{0\le \tau\le t} \|\bar{p}(\tau)\|_{L_\phi^2(U)}^{\nu_2}\Big).
\end{align*}

Noticing from Proposition \eqref{Linftee}  that $\nu_2\ge\nu_1$, 
we apply inequality \eqref{ee5} to $x= \|\bar{p}\|_{L_\phi^2(U)}$, $\beta=\nu_1$ and $\gamma=\nu_2$, and derive from the preceding inequality that 
\beq\label{presmall}
\|\bar{p}\|_{L^\infty(U\times(t/2,t))}
\le C  t^{-\frac{\kappa_1}{2-a}+\frac{\nu_1}2} N_1(0,t)^{\kappa_2} \Big( 1+  \sup_{0\le \tau\le t} \|\bar{p}(\tau)\|_{L_\phi^2(U)}^{\nu_2}\Big).
\eeq
Using estimate \eqref{EstLtwo} for $\|\bar{p}\|_{L_\phi^2(U)}$, and the fact $\mathcal{M}(t)\ge 1$, we obtain \eqref{smallt} from \eqref{presmall}.

Next, let $t\in [1,\infty)$. Applying  inequality \eqref{inftestDG} to $T_0=t-1$,  $T=1$ and $\theta=1/2$, and using \eqref{obo} again, we have 
\begin{align}
\|\bar{p}\|_{L^\infty(U\times(t-\frac12,t))}
&\le CN_1(t-1,t)^{\kappa_2}  \Big( \sup_{\tau\in[t-1,t]} \|\bar{p}(\tau)\|^{\nu_1}_{L_\phi^2(U)}+ \sup_{\tau\in[t-1,t]} \|\bar{p}(\tau)\|_{L_\phi^2(U)}^{\nu_2}\Big) \notag\\
&\le CN_1(t-1,t)^{\kappa_2}  \Big(1+ \sup_{\tau\in[t-1,t]} \|\bar{p}(\tau)\|_{L_\phi^2(U)}^{\nu_2}\Big). \label{forlim}
\end{align}

Again, using \eqref{EstLtwo} to estimate $\|\bar{p}\|_{L_\phi^2}$ in \eqref{forlim},  we obtain \eqref{larget}.

(ii) Taking the limit superior of \eqref{forlim}  as $t\to \infty$,  we have
\beqs
\limsup_{t\to \infty}\|\bar{p}\|_{L^\infty(U\times(t-\frac12,t))}
\le C \limsup_{t\to \infty}N_1(t-1,t)^{\kappa_2} \Big(1+ \limsup_{t\rightarrow\infty}\sup_{\tau\in[t-1,t]} \|\bar{p}(\tau)\|_{L_\phi^2(U)}^{\nu_2}\Big).
\eeqs
By the limit estimate \eqref{noinitial} and the fact $\mathcal{A}\ge 1$, we obtain \eqref{limpG}.

(iii) Using estimate \eqref{EstLtwo7} in \eqref{forlim}, we have  for sufficiently large $t$ that
\beq\label{forlim1}
\begin{aligned}
\|\bar{p}\|_{L^\infty(U\times(t-\frac12,t))}&
\le CN_1(t-1,t)^{\kappa_2}  \Big( 1+\mathcal B^{\frac 1{1-a}}+ \sup_{\tau\in[t-1,t]}G(\tau)^{\frac {2}{2-a}} \Big)^{\nu_2/2}.
\end{aligned}
\eeq
By Lemma \ref{Btt}, one has for $\tau\in[t-1,t]$ that 
\beqs 
G(\tau)\le G(t)+(t-\tau)(\mathcal B+1) \le  G(t)+\mathcal B+1.
\eeqs
Using this to estimate the sum on the right-hand side of \eqref{forlim1} gives
\beq\label{rhsB}
\begin{aligned}
&1+\mathcal B^{\frac 1{1-a}}+ \sup_{\tau\in[t-1,t]}G(\tau)^{\frac {2}{2-a}}
\le 1+\mathcal B^{\frac 1{1-a}}+ (G(t)+\mathcal B+1)^{\frac {2}{2-a}}\\
&\le C(1+\mathcal B^{\frac 1{1-a}}+\mathcal B^{\frac 2{2-a}}+G(t)^{\frac {2}{2-a}})
\le  C(\mathcal B^{\frac 1{1-a}}+G(t)^{\frac {2}{2-a}}).
\end{aligned}
\eeq
The last inequality uses \eqref{ee5} with $x=\mathcal B$, $\beta=2/(2-a)$ and $\gamma=1/(1-a)$, combined with the fact $G(t)\ge 1$.
Thus, desired estimate  \eqref{forlim11} follows \eqref{forlim1} and \eqref{rhsB}.
\end{proof}


\myclearpage
\section{Maximum estimates for the pressure's time derivative}\label{ptsec}

Let $p(x,t)$, $\Psi(x,t)$ and $\bar p(x,t)$ be as in section \ref{psec}. 
Define
 $$q(x,t)=p_t(x,t)\quad\text{and}\quad \bar{q}(x,t)=\bar{p}_t(x,t)=p_t(x,t)-\Psi_t.$$ 
We will estimate $L^\infty$-norm of $\bar q$. 


Assume (H2) again with  fixed number  $r>2$ and Sobolev constant $c_2$ in \eqref{P2}.

In the following, we also fix a number $r_2$ such that 
$$r_2> \frac{2}{2-r'}=\frac{2(r-1)}{r-2}.$$
Note that its conjugate exponent $r'_2$ belongs to $(1,2/r')$.

First, we establish a counter part of Proposition \ref{Linftee}.

\begin{proposition}\label{theo51}
There is a constant $\bar C>0$ independent of $c_2$ such that for any $T_0\ge 0$, $T>0$ and $\theta \in (0,1)$ we have 
\beq\label{qestdiG}
\begin{aligned}
\|\bar p_t\|_{L^\infty(U\times(T_0+\theta T,T_0+ T))}
&\le \bar C\max\{1,c_2\}^\frac{r}{r-2}\Big([(\theta T)^{-\frac12}\mathcal{S}_{T_0,T,\theta}]^\frac{1}{\delta_1}+(Z_{T_0,T}\mathcal{S}_{T_0,T,\theta})^\frac{1}{1+\delta_2} \Big)\\
&\quad\cdot\Big( \|\bar p_t\|_{L_\phi^2(U\times (T_0,T_0+T))}+ \|\bar p_t\|_{L_\phi^2(U\times (T_0,T_0+T))}^\frac{\delta_2}{1+\delta_2} \Big),
\end{aligned}
\eeq
where 
\begin{align*}
\delta_1&=1-\frac{r'} 2,  \quad \delta_2=\frac 1{ r'_2}-\frac{r'} 2,\\
\mathcal{S}_{T_0,T,\theta}&= \Big(B_1+\sup_{t\in[T_0+\theta T,T_0+T]}\int_U W_{1}(x) |\nabla p(x,t)|^{2-a}dx\Big)^\frac{ar'}{4(2-a)},\\
Z_{T_0,T}&=\|a_0(x)^{-1/2}\nabla \Psi_t\|_{L_\phi^{2 r_2}(U\times(T_0,T_0+T))} + T^{1/2}\|\Psi_{tt}\|_{L_\phi^{2 r_2}(U\times(T_0,T_0+T))}.
\end{align*}
\end{proposition}
\begin{proof}
Again, we  assume, without loss of generality, that $T_0=0$ and $\|\bar{p}_t\|_{L_\phi^2(U\times (0,T))}>0$.
The function $\bar{q}(x,t)$ solves 
\begin{align} \label{Qt}
\phi(x)\begin{displaystyle} \frac{\partial \bar{q}}{\partial t}\end{displaystyle}
&=\nabla \cdot(K(x,|\nabla p |)\nabla p)_t-\phi(x)\Psi_{tt} \quad \text{on } U\times(0,\infty), \\ 
\bar{q} &=0 \quad  \text{on }  \Gamma \times (0,\infty). \nonumber
\end{align}

We prove \eqref{qestdiG} by using De Giorgi's iteration for equation \eqref{Qt}.
Below,  $\bar C>0$ is generic, while $\bar C_3$, $\bar C_4$ have particular values, and they all depend on $a$, $r$, $r_2$, but not on $c_2$ in \eqref{P2}.

\textbf{Step 1.} Let $k\ge 0$ be arbitrary,  define $\bar{q}^{(k)}=\max\{\bar q-k,0\} $.
Denote by $\chi_k(x,t)$  the characteristic function of the set $\{(x,t) \in U\times(0,T): \bar q(x,t)> k\}$. 

Let $\zeta=\zeta(t)\ge 0$ be a $C^\infty$- function on $\mathbb R$ with $\zeta(t)=0$ on $(-\infty,0]$.

Multiplying \eqref{Qt} by the  function $\bar{q}^{(k)}\zeta^2$ and integrating over $U$, we have 
\beq\label{LR}
\int_U\frac{\partial \bar{q}}{\partial t}\bar{q}^{(k)}\zeta^2\phi dx=\int_U\nabla \cdot(K(x,|\nabla p |)\nabla p)_t\bar{q}^{(k)}\zeta^2dx-\int_U\Psi_{tt}\bar{q}^{(k)}\zeta^2\phi dx.
\eeq
The integrand on the left-hand side of \eqref{LR} is
$$\frac12\frac{\partial |\bar{q}^{(k)}|^2}{\partial t} \cdot \zeta^2 = \frac12\frac{\partial}{\partial t} (\bar{q}^{(k)}\zeta)^2
- |\bar{q}^{(k)}|^2\zeta'\zeta.$$ 
On the right-hand side of \eqref{LR}, we perform  integration by parts for the first term and using the fact  $\bar{q}=0$ on the boundary.
We obtain
\begin{align*}
\frac 12\frac{d}{d t}\int_U(\bar{q}^{(k)}\zeta)^2\phi dx
- \int_U|\bar{q}^{(k)}|^2\zeta'\zeta\phi dx=-\int_U (K(x,|\nabla p |)\nabla p)_t\nabla(\bar{q}^{(k)}\zeta^2)dx-\int_U\Psi_{tt}\bar{q}^{(k)}\zeta^2\phi dx.
\end{align*}
Since $\zeta$ depends on  $t$ only,  $\nabla(\bar{q}^{(k)}\zeta^2)=\nabla\bar{q}^{(k)}\zeta^2.$ 
Applying the product rule to the $t$-derivative in the first integral on the right-hand side gives 
\beq\label{Is}
\frac 12\frac{d}{d t}\int_U(\bar{q}^{(k)}\zeta)^2\phi dx
=\int_U|\bar{q}^{(k)}|^2\zeta'\zeta\phi dx
-I_1+I_2 -\int_U\Psi_{tt}\bar{q}^{(k)}\zeta^2\phi dx,
\eeq
where
\beq
I_1=\int_U K(x,|\nabla p |)\nabla q\cdot\nabla\bar{q}^{(k)}\zeta^2 dx,\quad I_2= -\int_U (K(x,|\nabla p |))_t\nabla p\cdot\nabla\bar{q}^{(k)}\zeta^2 dx.
\eeq

For $I_1$,
\begin{align*}
\nabla q \cdot \nabla\bar{q}^{(k)}
=\nabla \bar q\cdot \nabla\bar{q}^{(k)}+\nabla \Psi_t \cdot \nabla\bar{q}^{(k)}
=| \nabla\bar q^{(k)} |^2+\nabla \Psi_t \cdot \nabla\bar{q}^{(k)},
\end{align*}
Hence,
\beq\label{I2}
-I_1\le - \int_U K(x,|\nabla p |)|\nabla\bar{q}^{(k)}\zeta|^2dx
+\int_U K(x,|\nabla p |)|\nabla \Psi_t| | \nabla\bar{q}^{(k)}|\zeta^2 dx.
\eeq

For $I_2$,  by using \eqref{Kderiva}:
\beqs
| (K(x,|\nabla p |))_t\nabla p\cdot \nabla\bar{q}^{(k)}|
\le  |(K'(x,|\nabla p |)|\frac{|\nabla p\cdot\nabla q||\nabla p\cdot \nabla\bar{q}^{(k)}|}{|\nabla p|}
\le aK(x,|\nabla p |)|\nabla  q||\nabla\bar{q}^{(k)}|.
\eeqs
For the last product,
\beqs
\begin{aligned}
|\nabla q||\nabla\bar{q}^{(k)}|\le |\nabla \bar q||\nabla\bar q^{(k)}|+ |\nabla \Psi_t||\nabla\bar q^{(k)}|=|\nabla\bar q^{(k)}|^2+ |\nabla \Psi_t||\nabla\bar q^{(k)}|.
\end{aligned}
\eeqs
Hence,
\beq\label{I3}
|I_2|\le a\int_U K(x,|\nabla p |)|\nabla\bar{q}^{(k)}\zeta|^2dx
+a\int_U K(x,|\nabla p |)|\nabla \Psi_t| | \nabla\bar{q}^{(k)}|\zeta^2 dx.
\eeq
Then combining \eqref{Is}, \eqref{I2} and \eqref{I3} gives
\beq\label{q1}
-I_1+I_2\le  -(1-a)\int_U K(x,|\nabla p |)|\nabla\bar{q}^{(k)}\zeta|^2dx
 +(1+a)\int_U K(x,|\nabla p |)|\nabla \Psi_t| | \nabla\bar{q}^{(k)}|\zeta^2 dx.
\eeq


For the last integral, applying Cauchy's inequality gives
\begin{align*}
(1+a)K(x,|\nabla p |)|\nabla \Psi_t| | \nabla\bar{q}^{(k)}|
&\le \frac{1-a}2 K(x,|\nabla p |)| \nabla\bar{q}^{(k)}|^2+\frac{(1+a)^2}{2(1-a)}K(x,|\nabla p |)|\nabla \Psi_t| ^2\chi_k\\
\text{(by using \eqref{decK}) }&\le \frac{1-a}2 K(x,|\nabla p |)| \nabla\bar{q}^{(k)}|^2+\bar C a_0(x)^{-1}|\nabla \Psi_t| ^2\chi_k.
\end{align*}

Therefore, it follows this, \eqref{q1} and \eqref{Is} that
\begin{multline}\label{qtestt1}
\frac 12\frac{d}{d t}\int_U(\bar{q}^{(k)}\zeta)^2\phi dx
+\frac{1-a}2\int_U K(x,|\nabla p |)|\nabla(\bar q^{(k)}\zeta)|^2dx\\
 \le \int_U|\bar{q}^{(k)}|^2\zeta|\zeta'|\phi dx
 +\bar C\int_U a_0(x)^{-1}|\nabla \Psi_t|^2 \chi_k \zeta^2
 +\int_U|\Psi_{tt}| \bar{q}^{(k)} \zeta^2\phi dx.
\end{multline}

Let $\varepsilon>0$. Applying Cauchy's inequality to  the last integral, we have 
\begin{align*}
\int_U|\Psi_{tt}||\bar{q}^{(k)}|\zeta^2\phi dx
\le \varepsilon \int_U|\bar{q}^{(k)}\zeta|^2\phi dx+\frac1{4\varepsilon}\int_U |\Psi_{tt}|^2\chi_k\cdot \zeta^2\phi dx.
\end{align*}
Using this estimate in \eqref{qtestt1} we obtain 
\begin{multline}\label{barqestt2}
\frac 12\frac{d}{d t}\int_U(\bar{q}^{(k)}\zeta)^2\phi dx
+\frac{1-a}2\int_U K(x,|\nabla p |)|\nabla(\bar q^{(k)}\zeta)|^2dx \\ \le \int_U|\bar{q}^{(k)}|^2\zeta|\zeta'|\phi dx
+\varepsilon \int_U|\bar{q}^{(k)}\zeta|^2\phi dx
 +\bar C\int_U \Big(a_0(x)^{-1}|\nabla \Psi_t|^2+\varepsilon^{-1} |\Psi_{tt}|^2\Big)\chi_k\cdot \zeta^2\phi dx.
\end{multline}
Integrating \eqref{barqestt2} in time from $0$ to $t$ and then taking supremum on $(0,T)$, we find
\begin{multline*}
\sup_{0<t<T}\int_U|\bar{q}^{(k)}(x,t)\zeta(t)|^2\phi(x) dx+(1-a)\int_0^T\int_U K(x,|\nabla p |)|\nabla(\bar q^{(k)}\zeta)|^2dxdt\\
\le 4\int_0^T\int_U|\bar{q}^{(k)}|^2\zeta|\zeta'|\phi dx dt+4\varepsilon T \sup_{0<t<T} \int_U|\bar{q}^{(k)}(x,t)\zeta(t)|^2\phi(x) dx\\
+\bar C\int_0^T\int_U (a_0(x)^{-1}|\nabla \Psi_t|^2+\varepsilon^{-1}|\Psi_{tt}|^2)\chi_k\cdot \zeta^2\phi dxdt.
\end{multline*}
By selecting $\varepsilon=\frac{1}{8T}$, it follows that 
\begin{multline}\label{supineq4}
\sup_{0<t<T}\int_U|\bar{q}^{(k)}(x,t)\zeta(t)|^2\phi(x) dx+\int_0^T\int_U K(x,|\nabla p |)|\nabla(\bar q^{(k)}\zeta)|^2dxdt\\
\le \bar C\int_0^T\int_U|\bar{q}^{(k)}|^2\zeta|\zeta'|\phi dx dt
+\bar C\int_0^T\int_U E(x,t)\chi_k \cdot \zeta^2 \phi dx dt,
\end{multline} 
where
\beq
E(x,t)=a_0(x)^{-1}|\nabla \Psi_t|^2+ T|\Psi_{tt}|^2.
\eeq

 Applying inequality \eqref{gamandK2} of Corollary \ref{condAgamm} to $u=\bar{q}^{(k)}\zeta$ and $f(x,t)=|\nabla p(x,t)|$, we have 
\begin{multline}\label{basefit}
\| \bar{q}^{(k)}\zeta\|_{L_\phi^{4/r'}(U\times(0,T))}\\
 \le 2c_2^{\frac{r'} 2} \mathcal{S} \Big( \sup_{0<t<T} \int_U|\bar{q}^{(k)}(x,t)\zeta(t)|^2\phi(x) dx
 +\int_0^T\int_UK(x,|\nabla p(x,t)|)|\nabla \bar{q}^{(k)}\zeta|^{2}dxdt\Big)^\frac{1}{2},
\end{multline}
where 
\beq\label{SSdef}
\mathcal{S}= \left( \int_U  a_N(x)dx+ \sup_{t\in\supp\zeta\cap(0,T)} \int_U W_1(x)|\nabla p(x,t)|^{2-a}dx\right)^\frac{ar'}{4(2-a)}.
\eeq

Denote $c_3=c_2^{\frac{r'} 2}$.  Then combining \eqref{basefit} with \eqref{supineq4} yields
\beq\label{basefit1}
\| \bar{q}^{(k)}\zeta\|_{L_\phi^{4/r'}(U\times(0,T))}
\le \bar C c_3\mathcal{S}\left(\int_0^T\int_U|\bar{q}^{(k)}|^2\zeta|\zeta'|\phi dx dt+\int_0^T\int_U E\chi_k \cdot \zeta^2 \phi dx dt\right)^\frac 12.
\eeq

\textbf{Step 2.}  Let $M_0>0$ be a fixed value  that will be determined later. 
For $i \ge 0$, define 
\begin{align}
k_i= M_0(2-\frac{1}{2^i})\quad\text{and}\quad t_i=\theta T (1-\frac 1{2^{i+1}}).
\end{align}
Then the sequences $(k_i)_{i\ge 0}$ and $(t_i)_{i\ge 0}$ are strictly increasing with $M_0\le k_i <2M_0$, $\frac{\theta T}{2}\le t_i <\theta T$, and $\lim_{i\to\infty}k_i=2M_0$, $\lim_{i\to\infty}t_i=\theta T$. 

Denote $$A_{i,j}=\{(x,t): \bar{q}(x,t) > k_i, t\in(t_j,T)\}\quad \text{for} \quad i,j=0,1,2,3,\dots.$$

Let $\zeta_i(t)$ be a $C^\infty$-function on $\mathbb R$ valued in $[0,1]$ with the following properties
\beq\label{derivtt}
\zeta_i(t)=
\begin{cases}
0 \quad \text{in} \quad t\le t_i \\ 1 \quad \text{for} \quad t\ge t_{i+1},
\end{cases}
\quad\text{and}\quad
0\le \zeta_i'(t)\le \frac{2}{t_{i+1}-t_i}=\frac{2^{i+3}}{\theta T}.
\eeq

 Define $$F_i=\| \bar{q}^{(k_{i+1})}\zeta_i\|_{L_\phi^{4/r'}(A_{i+1, i})}=\| \bar{q}^{(k_{i+1})}\zeta_i\|_{L_\phi^{4/r'}(U\times(0,T))}.$$

Now, we apply inequality \eqref{basefit1} to $k=k_{i+1}$ and $\zeta=\zeta_i$.
Denote $\mathcal{S}_T=\mathcal{S}_{0,T,\theta}$. Then for all $i\ge 0$, the $\mathcal{S}$ defined in \eqref{SSdef} satisfies
$\mathcal{S}\le \mathcal{S}_T$.
Therefore, we have from \eqref{basefit1} that
\begin{align*}
F_i\le \bar C c_3\mathcal{S}_T\left(\int_0^T\int_U|\bar{q}^{(k_{i+1})}|^2\zeta_i \zeta_i' \phi dx dt+\int_0^T\int_UE\chi_{k_{i+1}} \zeta_i^2 \phi dx dt\right)^\frac 12.
\end{align*}
Using properties of $\zeta_i$ in  \eqref{derivtt}, we have  
\begin{align*}
F_i&\le \bar C c_3\mathcal{S}_T \Big(\frac{2^i}{\theta T}\int_0^T\int_U|\bar{q}^{(k_{i+1})}|^2\zeta_i\phi dx dt
+\iint_{A_{i+1,i}} E  \phi dx dt\Big)^\frac 12\\
&\le \bar C c_3\mathcal{S}_T \Big\{ \frac{2^\frac i2}{(\theta T)^{1/2}}\| \bar{q}^{(k_{i+1})}\|_{L_\phi^{2}(A_{i+1, i})}
+\Big(\iint_{A_{i+1,i}} E \phi dx dt\Big)^\frac 12\Big\}.
\end{align*}

For the last integral, applying H\"{o}lder's inequality with powers $ r_2$ and $ r'_2$, we derive 
\beq\label{Finew}
F_i
\le \bar C c_3\mathcal{S}_T \Big\{\frac{2^{\frac i2}}{(\theta T)^{1/2}}\| \bar{q}^{(k_{i+1})}\|_{L_\phi^{2}(A_{i+1, i})}+\Big(\iint_{A_{i+1,i}} E^{ r_2} \phi dx dt\Big)^\frac{1}{2 r_2}\Big(\iint_{A_{i+1,i}}  \phi dx dt\Big)^\frac{1}{2 r'_2}\Big\}.
\eeq

Let $Z_T=Z_{0,T}$ and recall that the measure $\bar \mu=\bar \mu_\phi$ is defined in \eqref{meas2} with $f(x,t)=\phi(x)$.
Clearly,
\beqs
\left(\iint_{A_{i+1,i}} E^{ r_2} \phi dx dt\right)^\frac{1}{2 r_2}\le \bar C Z_T
\quad\text{and}\quad
\| \bar{q}^{(k_{i+1})}\|_{L_\phi^{2}(A_{i+1, i})}\le \| \bar{q}^{(k_i)}\|_{L_\phi^{2}(A_{i, i})}.
\eeqs
Hence, we derive from \eqref{Finew} that
\beq\label{Fi2}
F_i
\le \bar C c_3\mathcal{S}_T\Big\{ 2^i(\theta T)^{-\frac 12}\| \bar{q}^{(k_i)}\|_{L_\phi^{2}(A_{i, i})}+Z_T \bar \mu(A_{i+1,i})^\frac{1}{2 r'_2}\Big\}.
\eeq

Next, by H\"older's inequality and, again,  the fact $A_{i+1,i+1}\subset A_{i+1,i}$ one has
\beq\label{brqhold1}
\begin{aligned}
\|\bar{q}^{(k_{i+1})}\|_{L_\phi^{2}(A_{i+1,i+1})} &\le\|\bar{q}^{(k_{i+1})}\|_{L_\phi^{4/r'}(A_{i+1,i+1})}\bar \mu(A_{i+1,i+1})^{\frac 12-\frac{r'}4}\\
&=\|\bar{q}^{(k_{i+1})}\zeta_i\|_{L_\phi^{4/r'}(A_{i+1,i+1})}\bar \mu(A_{i+1,i+1})^{\frac 12-\frac{r'}4}
\le F_i \bar \mu(A_{i+1,i})^{\frac 12-\frac {r'}4}.
\end{aligned}
\eeq

Then \eqref{Fi2} and \eqref{brqhold1} imply
\begin{multline}\label{pmA}
\|\bar{q}^{(k_{i+1})}\|_{L_\phi^{2}(A_{i+1,i+1})} 
\le \bar C c_3\mathcal{S}_T\Big({2^i}(\theta T)^{-\frac 12}\| \bar{q}^{(k_i)}\|_{L_\phi^{2}(A_{i, i})}\bar \mu(A_{i+1,i})^{\frac 12-\frac {r'}4}\\
+Z_T \bar \mu(A_{i+1,i})^{\frac 12-\frac {r'}4+\frac{1}{2 r'_2}}\Big).
\end{multline}

To estimate the measure $\bar \mu(A_{i+1,i})$, note that
\begin{align*}
\|\bar{q}^{(k_i)}\|_{L_\phi^2(A_{i,i})} \ge \|\bar{q}^{(k_i)}\|_{L_\phi^2(A_{i+1,i})} \ge (k_{i+1}-k_i)\bar \mu(A_{i+1,i})^{\frac 12}.
\end{align*}
Hence 
\beq\label{Aiiesttt1}
\bar \mu(A_{i+1,i}) \le  (k_{i+1}-k_i)^{-2}\|\bar{q}^{(k_i)}\|^2_{L_\phi^2(A_{i,i})} 
= \frac{4^{i+1}}{M_0^2}\|\bar{q}^{(k_i)}\|^2_{L_\phi^2(A_{i,i})}.
\eeq

Now substituting  \eqref{Aiiesttt1} into \eqref{pmA}, we obtain 
\beq\label{brqhold2}
\begin{aligned}
\|\bar{q}^{(k_{i+1})}\|_{L_\phi^{2}(A_{i+1,i+1})} &\le  \bar C c_3 {4^i}\mathcal{S}_T\Big\{ (\theta T)^{-\frac 12}M_0^{-1+\frac{r'}{2}}\|\bar{q}^{(k_i)}\|_{L_\phi^2(A_{i,i})}^{ 2-\frac {r'}2}\\
&\quad +  Z_T  M_0^{-1+\frac{r'}{2}-\frac{1}{ r'_2}} \|\bar{q}^{(k_i)}\|_{L_\phi^2(A_{i,i})}^{1-\frac{r'} 2+\frac{1}{ r'_2}}\Big\}.
\end{aligned}
\eeq

\textbf{Step 3.} Defining $Y_i=\|\bar{q}^{(k_{i})}\|_{L_\phi^{2}(A_{i,i})}=\|\bar{q}^{(k_{i})}\|_{L_\phi^{2}(U\times(t_i,T))}$ for $i\ge 0$. From \eqref{brqhold2}, we have for all $i\ge 0$ that 
\beq\label{brqhold3}
\begin{aligned}
Y_{i+1} \le  4^i(D_1Y_i^{ 1+\delta_1} + D_2Y_i^{1+\delta_2}),
\end{aligned}
\eeq
where $D_1=\bar C_3 c_3\mathcal{S}_T(\theta T)^{-\frac 12}M_0^{-\delta_1}$ and $D_2=\bar C_3 c_3\mathcal{S}_TZ_TM_0^{-1-\delta_2}$, for some $\bar C_3>0$.

Applying Lemma \ref{multiseq} to the sequence $\{Y_i\}_{i=0}^\infty$ and \eqref{brqhold3} with $M_0$ chosen sufficiently large such that 
\beq\label{Yii}
\begin{aligned}
Y_0\le \bar C_4\min \{D_1^{-\frac{1}{\delta_1}},\,D_2^{-\frac{1}{\delta_2}}\}
\end{aligned}
\eeq
for a particular $\bar C_4>0$, 
then $\lim_{i\to\infty}Y_i=0$, i.e., 
\beq\label{qm0}
\lim_{i\to\infty}\|\bar{q}^{(k_i)}\|_{L_\phi^2(A_{i,i})}=\int_{\theta T}^T\int_U |\bar{q}^{(2M_0)}|^2 \phi dxdt=0.
\eeq
Using the same arguments that yield \eqref{pM} from \eqref{pm0}, here we have from \eqref{qm0} that
\beq \label{qM}
|\bar{q}(x,t)|\le 2M_0\quad \text{a.e. in }U\times (\theta T,T).
\eeq

Lastly, we find  $M_0> 0$ to satisfy \eqref{Yii}. Note that $Y_0\le \|\bar{q}\|_{L_\phi^2(U\times (0,T))} $.  
Then a sufficient condition for  \eqref{Yii} is
$$\|\bar{q}\|_{L_\phi^2(U\times (0,T))}\le \bar C D_1^{-\frac 1{\delta_1}}\text{ and }\|\bar{q}\|_{L_\phi^2(U\times (0,T))}\le \bar C D_2^{-\frac 1{\delta_2}}.$$
Solving these inequalities gives
\begin{align*}
M_0\ge \bar C(c_3\mathcal{S}_T(\theta T)^{-1/2})^\frac 1{\delta_1} \|\bar{q}\|_{L_\phi^2(U\times (0,T))} \text{ and } 
M_0\ge \bar C(c_3\mathcal{S}_TZ_T)^{\frac{1}{1+\delta_2}} \|\bar{q}\|_{L_\phi^2(U\times (0,T))}^\frac{\delta_2}{1+\delta_2}.
\end{align*}

Since $1+\delta_2>\delta_1$, we estimate the $c_3$-terms by 
\beqs
c_3^\frac1{\delta_1}+c_3^\frac1{1+\delta_2}\le 2\max\{1,c_3\}^\frac1{\delta_1}=2\max\{1,c_2\}^\frac{r'}{2-r'}=2\max\{1,c_2\}^\frac{r}{r-2}.
\eeqs

Hence we choose $M_0$ as 
\begin{align*}
M_0=\bar C \max\{1,c_2\}^\frac{r}{r-2} \Big( (\mathcal{S}_T(\theta T)^{-1/2})^\frac 1{\delta_1}+(\mathcal{S}_TZ_T)^{\frac{1}{1+\delta_2}}\Big) \Big( \|\bar{q}\|_{L_\phi^2(U\times (0,T))}+ \|\bar{q}\|_{L_\phi^2(U\times (0,T))}^\frac{\delta_2}{1+\delta_2} \Big).
\end{align*}

Then inequality \eqref{qestdiG} follows \eqref{qM}. The proof is complete.
\end{proof}

Then combining Proposition \ref{theo51} with estimates in section \ref{review} yields specific bounds for $\bar p_t$.

For $t>s\ge 0$, define
\beq
N_2(s,t)=1+\|a_0^{-1/2}\nabla \Psi_t\|_{L_\phi^{2 r_2}(U\times(s,t))}+\|\Psi_{tt}\|_{L_\phi^{2 r_2}(U\times(s,t))}.
\eeq
Then $N_2(s,t)\ge 1$ and 
\beq\label{ZoT}
1+Z_{T_0,T}\le (\max\{1,T\})^{1/2}N_2(T_0,T_0+T).
\eeq

Below, we assume (H1) and denote by $C$ a generic positive constant depending on $a$, $r$, $r_2$,  $c_1$  in \eqref{P1}, and $c_2$ in \eqref{P2}.

\begin{theorem}\label{theo52}
Let $\delta_1$, $\delta_2$ be as in Proposition \ref{theo51}.

\begin{enumerate}
\item[\rm(i)] If $t\in(0,\frac 32)$ then
\beq\label{pt1}
\|\bar p_t\|_{L^\infty(U\times(t/2,t))}\le C t^{-\frac 1{2\delta_1}} N_2(0,t)^\frac{1}{1+\delta_2} \left( A_0+ \mathcal M(t)^\frac2{2-a}+\int_0^t G_1(\tau)d\tau \right)^{\kappa_4},
\eeq
where 
\beqs 
\kappa_4=\frac12+\frac{ar}{2(2-a)(r-2)},
\eeqs
\beqs
A_0=\int_U H(x,|\nabla p(x,0)|)dx+ \int_U \bar{p}^2(x,0)\phi  dx.
\eeqs

If $t\ge \frac 32$ then
\beq\label{pt2}
\|\bar p_t\|_{L^\infty(U\times(t-\frac14,t))}
\le C N_2(t-\tfrac12,t)^\frac{1}{1+\delta_2} \Big( \|\bar{p}(0)\|_{L^2_\phi}^2+ \mathcal M(t)^{\frac 2{2-a}} +\int_{t-\frac54}^t G_1(\tau)d\tau\Big)^{\kappa_4}.
\eeq

\item[\rm(ii)] If $\mathcal A<\infty$ then 
\begin{multline}\label{liminf}
\limsup_{t\to\infty} \| \bar p_t\|_{L^\infty(U\times(t-\frac14,t))}
\le C \big(\limsup_{t\to\infty}N_2(t-\tfrac12,t)\big)^{\frac{1}{1+\delta_2}} \Big(\mathcal A^\frac2{2-a}+\limsup_{t\to\infty}\int_{t-1}^t G_1(\tau)d\tau\Big)^{\kappa_4}.
\end{multline}

\item[\rm(iii)] If $\mathcal B<\infty$ then  there is $T>0$ such that  for all $t>T$,
\beq\label{qestdiB5}
\|\bar p_t\|_{L^\infty(U\times(t-\frac14,t))}
\le C N_2(t-\tfrac12,t)^\frac{1}{1+\delta_2} \Big( \mathcal B^{\frac 1{1-a}}+G(t)^{\frac 2{2-a}}+\int_{t-\frac54}^t G_1(\tau)d\tau\Big)^{\kappa_4}.
\eeq
\end{enumerate}
\end{theorem}
\begin{proof}
Similar to the proof of Theorem \ref{mainp}, all constants $C$'s in section \ref{review} are made dependent of $c_1$ defined by \eqref{newc1}.

(i) Let $t\in (0,3/2)$.
We apply \eqref{qestdiG} to $t\in (0,3/2]$, $T_0=0$, $T=t$ and $\theta=1/2$.
On the one hand, the quantity  $\mathcal{S}_{T_0,T,\theta}=\mathcal{S}_{0,t,1/2}$ in \eqref{qestdiG} is bounded by using \eqref{gradpW1}:
\beqs
\mathcal{S}_{0,t,1/2}=\Big(B_1+\sup_{\tau\in[t/2,t]} \int_U W_1(x)|\nabla p(x,\tau)|^{2-a} dx\Big)^\frac{ar'}{4(2-a)} \le C S_1(t)^\frac{ar'}{4(2-a)},
\eeqs
where
$$S_1(t)= A_0+\mathcal M(t)^\frac2{2-a} +\int_0^t G_1(\tau)  d\tau\ge 1.$$
Above, we used the facts $\mathcal{M}(t)$ is increasing and $\mathcal{M}(t)\ge 1,B_1$.
On the other hand, to estimate $Z_{T_0,T}=Z_{0,t}$ in \eqref{qestdiG}, we use \eqref{ZoT}. These estimates result in
\beq\label{qest11}
\begin{aligned}
\|\bar{q}\|_{L^\infty(U\times(t/2,t))}
&\le C\Big( t^{-\frac 1{2\delta_1}} S_1(t)^{\frac1{\delta_1}\frac{ar'}{4(2-a)}}+ N_2(0,t)^\frac{1}{1+\delta_2} S_1(t)^{\frac1{1+\delta_2}\frac{ar'}{4(2-a)}}\Big)\\
&\quad \cdot \Big(  \|\bar{q}\|_{L_{\phi}^2(U\times(0,t))}+\|\bar{q}\|_{L_{\phi}^2(U\times(0,t))}^\frac{\delta_2}{1+\delta_2} \Big).
\end{aligned}
\eeq

Note that $1+\delta_2=\delta_1+1/r_2'>\delta_1$. Hence, the maximum power of $S_1(t)$ is
\beqs
\kappa_5=\frac{ar'}{4(2-a)}\cdot \frac1{\delta_1}=\kappa_4-\frac12.
\eeqs

For the power $\|\bar{q}\|_{L_{\phi}^2(U\times(0,t))}$ in \eqref{qest11}, note that $\frac{\delta_2}{1+\delta_2}\le 1$. Then it follows
\beqs
\begin{aligned}
\|\bar{q}\|_{L^\infty(U\times(t/2,t))}&\le Ct^{-\frac 1{2\delta_1}} N_2(0,t)^\frac{1}{1+\delta_2} S_1(t)^{\kappa_5}\Big( 1+ \|\bar{q}\|_{L_{\phi}^2(U\times(0,t))} \Big).
\end{aligned}
\eeqs

To estimate $\|\bar{q}\|_{L_{\phi}^2(U\times(0,t))}$, we integrate  \eqref{ptrowest} in time from $0$ to $t$, and have
\beqs
\int_0^t\int_U\bar q^2(x,\tau)\phi(x) dx d\tau\le  A_0 +C\int_0^t (G(\tau)+G_1(\tau))d\tau\le C S_1(t).
\eeqs
Then
\beqs
\|\bar{q}\|_{L^\infty(U\times(t/2,t))}
\le Ct^{-\frac 1{2\delta_1}} N_2(0,t)^\frac{1}{1+\delta_2} S_1(t)^{\kappa_5} ( 1+S_1(t) )^{1/2}
\le Ct^{-\frac 1{2\delta_1}} N_2(0,t)^\frac{1}{1+\delta_2} S_1(t)^{\kappa_5+1/2},
\eeqs
and inequality  \eqref{pt1} follows.

Now,  let $t \ge \frac 32$. Using \eqref{qestdiG}  with $T_0=t-\frac 12$, $T=\frac 12$ and $\theta=1/2$, and utilizing \eqref{ZoT}, we have 
\begin{align*}
\|\bar{q}\|_{L^\infty(U\times(t-\frac14,t))}
&\le C\Big(S_2(t)^\frac{1}{\delta_1}+S_2(t)^\frac{1}{1+\delta_2}N_2(t-1/2,t)^\frac{1}{1+\delta_2}\Big)\Big( \|\bar{q}\|_{L_{\phi}^2(U\times (t-\frac 12,t))}+ \|\bar{q}\|_{L_{\phi}^2(U\times (t-\frac 12,t))}^\frac{\delta_2}{1+\delta_2} \Big)\\
&\le C N_2(t-1/2,t)^\frac{1}{1+\delta_2} \Big(S_2(t)^\frac{1}{\delta_1}+S_2(t)^\frac{1}{1+\delta_2}\Big)\Big( 1+\|\bar{q}\|_{L_{\phi}^2(U\times (t-\frac 12,t))} \Big),
\end{align*}
where
\beqs
S_2(t)= \left(B_*+\sup_{\tau\in[t-1/4,t]}\int_U W_{1}(x) |\nabla p(x,\tau)|^{2-a}dxd\tau\right)^\frac{ar'}{4(2-a)}.
\eeqs

Since $1+\delta_2>\delta_1$ and $S_2(t)\ge 1$, it follows that
\beq\label{qestdiG11}
\|\bar{q}\|_{L^\infty(U\times(t-\frac14,t))}
\le C N_2(t-1/2,t)^\frac{1}{1+\delta_2} S_2(t)^\frac{1}{\delta_1}\Big( 1+\|\bar{q}\|_{L_{\phi}^2(U\times (t-\frac 12,t))} \Big).
\eeq

Using \eqref{C1} to estimate $S_2(t)$, and using \eqref{Gtwo}, \eqref{EstLtwo} to estimate $\|\bar{q}\|_{L_{\phi}^2(U\times (t-\frac 12,t))}$ in \eqref{qestdiG11} yield
\begin{align*}
&\|\bar{q}\|_{L^\infty(U\times(t-\frac14,t))}
\le C N_2(t-1/2,t)^\frac{1}{1+\delta_2}\\
&\quad \cdot \Big(\|\bar{p}(0)\|^2_{L_\phi^2(U)}+ \mathcal M(t-1)^{\frac 2{2-a}} +\sup_{\tau\in[t-1/4,t]}\mathcal M(\tau)^{\frac 2{2-a}} +
\int_{t-\frac54}^t (G(\tau)+G_1(\tau))d\tau\Big)^{\kappa_5+\frac 12}.
\end{align*}

Since $\mathcal{M}(t)\ge M(\tau)\ge  G(\tau)\ge 1$ for all $\tau\in[t-\frac54,t]$,  we obtain \eqref{pt2} consequently.

(ii) Taking limit superior  of \eqref{qestdiG11}, we have 
\beq
\limsup_{t\to\infty}\|\bar{q}\|_{L^\infty(U\times(t-\frac14,t))}
\le C \limsup_{t\to\infty} N_2(t-1/2,t)^\frac{1}{1+\delta_2} \limsup_{t\to\infty} S_2(t)^\frac{1}{\delta_1}
\Big( 1+\limsup_{t\to\infty}\|\bar{q}\|_{L_{\phi}^2(U\times (t-\frac 12,t))} \Big).
\eeq

Using \eqref{C3} to estimate the limit superior of $S_2(t)$, and using  \eqref{Gtwo}, \eqref{noinitial} to estimate the limit superior of $\|\bar{q}\|_{L_{\phi}^2(U\times (t-\frac 12,t))}$, we obtain
\begin{multline}\label{qestdiG122}
\limsup_{t\to\infty}\|\bar{q}\|_{L^\infty(U\times(t-\frac14,t))}
\le C \limsup_{t\to\infty} N_2(t-1/2,t)^\frac{1}{1+\delta_2} \Big(\mathcal A^\frac 2{2-a} + \limsup_{t\to\infty}\int_{t-1}^t G_1(\tau) d\tau\Big)^{\kappa_5}\\
 \cdot \Big(\mathcal A^\frac 2{2-a}+\limsup_{t\to\infty}\int_{t-1}^t (G(\tau) + G_1(\tau))d\tau \Big)^{\frac 12}.
\end{multline}

Note that
\beqs
\limsup_{t\to\infty}\int_{t-1}^t G(\tau) d\tau\le \limsup_{t\to\infty}G(t)=\mathcal A
\le \mathcal A^\frac2{2-a}.
\eeqs
Then \eqref{qestdiG122} implies 
\beqs
\limsup_{t\to\infty}\|\bar{q}\|_{L^\infty(U\times(t-\frac14,t))}
\le C \limsup_{t\to\infty} N_2(t-1/2,t)^\frac{1}{1+\delta_2} \Big(\mathcal A^\frac 2{2-a} + \limsup_{t\to\infty}\int_{t-1}^t G_1(\tau) d\tau\Big)^{\kappa_5+\frac 12},
\eeqs
and \eqref{liminf} follows.

(iii) Using \eqref{C4} to estimate $S_2(t)$, and combining \eqref{Gtwo} with \eqref{EstLtwo7} to estimate $\|\bar{q}\|_{L_{\phi}^2(U\times (t-\frac 12,t))}$ in \eqref{qestdiG11}, we have for sufficient large $t$ that
\begin{multline*}
\|\bar{q}\|_{L^\infty(U\times(t-\frac14,t))}
\le C N_2(t-1/2,t)^\frac{1}{1+\delta_2} \sup_{s\in [t-1/4,t]} \Big(B_*+\mathcal B^{\frac 1{1-a}}+ G(s)^{\frac 2{2-a}}+\int_{s-1}^s G_1(\tau)d\tau\Big)^{\kappa_5}\\
 \cdot \Big(1+\mathcal B^{\frac 1{1-a}}+ G(t-1)^{\frac 2{2-a}}+\int_{t-1}^t (G(\tau)+G_1(\tau))d\tau\Big)^\frac 12.
\end{multline*}

Since $G(t)\ge B_*\ge 1$, we have
\beq\label{qB}
\|\bar q(t)\|_{L^\infty(U\times(t-\frac14,t))}
\le C N_2(t-1/2,t)^\frac{1}{1+\delta_2} \Big( \mathcal B^{\frac 1{1-a}}+\sup_{\tau\in [t-\frac54,t]} G(\tau)^{\frac 2{2-a}}+\int_{t-\frac54}^t G_1(\tau)d\tau\Big)^{\kappa_5+\frac 12}.
\eeq
We apply Lemma \ref{Btt} to  estimate $G(\tau)$
for $\tau\in [t-\frac54,t]$ in terms of $\mathcal{B}$  as 
$$G(\tau)\le G(t)+\frac 54(\mathcal B+1).$$
Similar to the proof of Theorem \ref{mainp}, we then obtain \eqref{qestdiB5} from \eqref{qB}.
\end{proof}

\myclearpage

\textbf{Acknowledgement.}  The authors would like to thank Phuc Nguyen for helpful discussions. LH acknowledges the support by NSF grant DMS-1412796.


\myclearpage




 \bibliographystyle{abbrv}
\def\cprime{$'$}

 \end{document}